\documentclass[12pt]{article}
\usepackage{amsfonts}

\usepackage{bm}
\usepackage{enumerate} 
\usepackage{amssymb,amsmath}
\usepackage{amsthm} 
\usepackage{caption}
\usepackage{subcaption}
\usepackage{accents} 

\usepackage{dsfont}

\usepackage{stackengine}

\usepackage{accents}

\usepackage{tikz}
\usetikzlibrary{automata,topaths}
\usetikzlibrary{shapes}
\usetikzlibrary{plotmarks}

\usepackage{xr} %

\usepackage{float} 
\usepackage{textcomp} 
\usepackage{siunitx}
\usepackage{color} 
\definecolor{lightblue}{rgb}{0,0.2,0.5}
\usepackage[colorlinks=true, urlcolor=lightblue,linkcolor=lightblue, citecolor=lightblue]{hyperref}
\usepackage{ucs}

\definecolor{dblackcolor}{rgb}{0.0,0.0,0.0}
\definecolor{dbluecolor}{rgb}{0.01,0.02,0.7}
\definecolor{dgreencolor}{rgb}{0.2,0.4,0.0}
\definecolor{dgraycolor}{rgb}{0.30,0.3,0.30}

\usepackage{xcolor}
\definecolor{ForestGreen}{RGB}{34,139,34}
\definecolor{mauve}{rgb}{0.7,0,0.43}
\definecolor{dkgreen}{rgb}{0,0.6,0}
\definecolor{darkgreen}{rgb}{0,0.6,0}
\definecolor{darkorange}{rgb}{1.0, 0.55, 0.0}
\definecolor{lightblue}{rgb}{0,0.2,0.5}
\definecolor{blue1}{rgb}{0,0.1,0.9}

\definecolor{lightblue}{rgb}{0,0.2,0.5}

\usepackage{listings}
\lstdefinelanguage{Maple}{
    morekeywords={proc, if, return, map, op, int, for, do, local, nops, convert, end},
    sensitive=false, %
    morecomment=[l]{//}, %
    morecomment=[s]{/*}{*/}, %
    morestring=[b]" %
} %

\usepackage{graphicx}
\usepackage{flushend,cuted}
\usepackage{bm}
\usepackage{tabularx}
\usepackage{indentfirst}
\usepackage{amssymb}
\usepackage{xparse}
\usepackage{tikz}
\usepackage{mdwlist}
\usepackage{tkz-graph}

\DeclareMathAlphabet{\eufrak}{U}{}{}{} 
\SetMathAlphabet\eufrak{normal}{U}{euf}{m}{n}
\SetMathAlphabet\eufrak{bold}{U}{euf}{b}{n}

\newcommand{\R}{\mathbb{R}}

\newcommand{\C}{\mathcal{C}}
\newcommand{\E}{\mathbb{E}}

\newcommand{\N}{\mathbb{N}}
\newcommand{\inte}{\mathbb{N}}

\makeatletter
\newcommand*\rel@kern[1]{\kern#1\dimexpr\macc@kerna}
\newcommand*\widebar[1]{%
  \begingroup
  \def\mathaccent##1##2{%
    \rel@kern{0.8}%
    \overline{\rel@kern{-0.8}\macc@nucleus\rel@kern{0.2}}%
    \rel@kern{-0.2}%
  }%
  \macc@depth\@ne
  \let\math@bgroup\@empty \let\math@egroup\macc@set@skewchar
  \mathsurround\z@ \frozen@everymath{\mathgroup\macc@group\relax}%
  \macc@set@skewchar\relax
  \let\mathaccentV\macc@nested@a
  \macc@nested@a\relax111{#1}%
  \endgroup
}
\makeatother

\makeatletter
\DeclareRobustCommand\widecheck[1]{{\mathpalette\@widecheck{#1}}}
\def\@widecheck#1#2{%
    \setbox\z@\hbox{\m@th$#1#2$}%
    \setbox\tw@\hbox{\m@th$#1%
       \widehat{%
          \vrule\@width\z@\@height\ht\z@
          \vrule\@height\z@\@width\wd\z@}$}%
    \dp\tw@-\ht\z@
    \@tempdima\ht\z@ \advance\@tempdima2\ht\tw@ \divide\@tempdima\thr@@
    \setbox\tw@\hbox{%
       \raise\@tempdima\hbox{\scalebox{1}[-1]{\lower\@tempdima\box
\tw@}}}%
    {\ooalign{\box\tw@ \cr \box\z@}}}
\makeatother

\newtheorem{prop}{Proposition}[section]
\newtheorem{assumption}[prop]{Assumption} %
\newtheorem{lemma}[prop]{Lemma}
\newtheorem{definition}[prop]{Definition}
\newtheorem{corollary}[prop]{Corollary}
\newtheorem{theorem}[prop]{Theorem}
\newtheorem{remark}[prop]{Remark}
\newtheorem{example}[prop]{Example}

\usepackage[title]{appendix} %

\def\({\left(}
\def\){\right)}

\def\[{\left[}
\def\]{\right]}
\def\real{{\mathord{\mathbb R}}}
\def\N{{\mathord{\mathbb N}}}
\newcommand{\id}{\textbf{\rm Id}}

\def\P{\mathbb{P}}

\newcommand{\e}{\varepsilon}

\makeatletter
\newcommand*\bigcdot{\mathpalette\bigcdot@{.5}}
\newcommand*\bigcdot@[2]{\mathbin{\vcenter{\hbox{\scalebox{#2}{$\m@th#1\bullet$}}}}} %
\makeatother

\usepackage{cprotect}

\usepackage{empheq}

\usepackage{algorithm}
\usepackage{algorithmic}

\usepackage{array}
\usepackage{relsize} 

\newcolumntype{C}[1]{>{\centering\let\newline\\\arraybackslash\hspace{0pt}}m{#1}}

\usepackage{comment}

\newcommand{\F}{\mathcal F}

\newcommand{\ind}{\mathbf 1}

\newcommand{\pt}{\partial}

\newenvironment{Proof}{\removelastskip\par\medskip
\noindent{\em Proof.} \rm}{\penalty-20\null\hfill$\square$\par\medbreak}

\allowdisplaybreaks

\numberwithin{equation}{section}

\GraphInit[vstyle = Shade]
\usetikzlibrary[intersections,
positioning,
petri,
backgrounds,
fit,
decorations.pathmorphing,
arrows,
arrows.meta,
bending,
calc,
intersections,
through,
backgrounds,
shapes.geometric,
quotes,
matrix,
trees,
shapes.symbols,
graphs,
math,
patterns,
external,
scopes,
matrix,
lindenmayersystems,
shapes.callouts,
shapes.misc,
angles,
shapes.arrows,
shadings]

\begin{document}
\title{
\huge
 Probabilistic representation of ODE solutions with quantitative estimates %
} 

\author{
  Qiao Huang\footnote{
 School of Mathematics,
  Southeast University,
  Nanjing 211189,
  P.R. China,
  and 
 School of Physical and Mathematical Sciences, 
 Nanyang Technological University, 
 21 Nanyang Link, Singapore 637371. 
 \href{mailto:qiao.huang@seu.edu.cn}{qiao.huang@seu.edu.cn}}
 \qquad
 Nicolas Privault\footnote{%
 School of Physical and Mathematical Sciences, 
 Nanyang Technological University, 
 21 Nanyang Link, Singapore 637371.
 \href{mailto:nprivault@ntu.edu.sg}{nprivault@ntu.edu.sg}
}
}

\maketitle

\vspace{-0.5cm}

\begin{abstract} 
 This paper considers the probabilistic representation of the solutions of
 ordinary differential equations (ODEs) by the generation of marked
 random trees in which marks can be interpreted 
 as mutant types in population genetics models.
 We present sufficient conditions 
 on equation coefficients that ensure the 
 integrability and uniform integrability of the 
 functionals of random trees used in this representation.
 Those conditions rely on the analysis 
 of a marked branching process that controls the growth of random trees 
 and provide implicit lower bounds on
 the explosion time of the underlying ODE,
 thus providing a connection between branching process explosion 
 and the existence and uniqueness of ODE solutions. 
\end{abstract}


\baselineskip0.7cm

\section{Introduction}
\noindent
 Consider the $d$-dimensional autonomous ODE problem %
 \begin{equation}
   \label{ODE}
  \left\{
  \begin{array}{ll} 
    x'(t) = f(x(t)), & t \in(0,T],
      \medskip
      \\
    x(0) = x_0\in\R^d, &
  \end{array}
  \right. 
\end{equation}
 where $f\in {\cal C}^\infty (\R^d  , \R^d)$ and $T>0$.
 Since $f$ is smooth, the Picard--Lindel\"of theorem guarantees
 local existence and uniqueness of the solution of the ODE~\eqref{ODE}. 

 \medskip

 It is well known that the solutions of ordinary differential
 equations (ODEs) such as \eqref{ODE} can be represented
 using Butcher trees,
 by combining rooted tree enumeration with Taylor expansions. 
 For this, if $x(t)$ is sufficiently smooth at $t=t_0$, 
 consider the Taylor expansion 
\begin{align} 
  \nonumber %
  x(t) & = x_0 + (t-t_0) f(x_0) + \frac{(t-t_0)^2}{2} \big(
  \nabla f(f) \big) (x_0)
  \\
\label{Taylor}
  & \quad + \frac{(t-t_0)^3}{6} \big( \nabla^2 f(f,f) + \nabla f(\nabla f (f)) \big) (x_0) + \cdots, 
\end{align}
 which uses the ``elementary differentials'' 
$$
f, \ \nabla f(f), \ \nabla^2 f(f,f), \ \nabla f(\nabla f (f)), \ldots. 
$$
see \eqref{aa1}-\eqref{aa2} below for their
 componentwise expressions. 
 From \cite{butcher1963,butcherbk},
 it is known that the expansion \eqref{Taylor} admits the formulation 
 \begin{equation}
   \label{fjkld9} 
 x(t) = \sum_{\tau} \frac{(t-t_0)^{|\tau|}}{\tau!\zeta (\tau)} F(\tau)(x_0),
 \end{equation}
 as a summation %
 over rooted trees $\tau$, where the functional 
 $F(\tau)$ represents the elementary differentials, 
 and $\zeta (\tau )$, $\tau !$ denote respectively the symmetry and
 factorial of the tree $\tau$. 
 The series \eqref{fjkld9} can be used to estimate
 ODE solutions by expanding $x(t)$ into a sum over trees
 up to a finite order,
 see e.g. \cite[Chapters~4-6]{deuflhard} %
 and references therein.

 \medskip

 In \cite{penent4}, stochastic branching
 processes have been used to express 
 ODE solutions as the expected value
 of a functional of random trees. 
 In \cite{huangprivault}, this approach has been
 interpreted as a Monte Carlo random
 generation of the Butcher trees $\tau$
 for the numerical estimation of the series
 \eqref{fjkld9}. 
 Monte Carlo estimators represent an alternative to the truncation
 of series, and they allow for estimates that improve
 via successive iterations.

 \medskip
  
 Those results also complement 
 other approaches to the use of stochastic processes
 to provide a diffusion interpretation for
 the solutions of partial differential equations
 via the Feynman--Kac formula \cite{Feyn2},
 and more generally in the fully nonlinear case 
 via stochastic branching mechanisms or stochastic cascades,
 see, e.g., 
 \cite{skorohodbranching},
 \cite{inw},
 \cite{hpmckean}, 
 \cite{sznitman},
 \cite{waymire},
 \cite{labordere}, 
 \cite{penent2022fully}.

 \medskip

 However, the above references generally 
 assume the uniform integrability of 
 random weights and/or the existence of a solution, see 
 \cite[Theorem~3.5]{labordere},
 \cite[Theorem~4.2]{penent2022fully}, 
 \cite[Theorem~4.2]{penent4}.
 This paper studies the stability of stochastic branching algorithms
 for the estimation of ODE solutions by probabilistic methods,
 without making such assumptions.

 \medskip
 
 On the one hand, in Theorem~\ref{tjklf-3-2}
 we derive integrability
 conditions on random weights that ensure the
 existence of the solution of
 the ODE~\eqref{ODE}
 together with the validity of
 its probabilistic representation.
 In Theorem~\ref{tjklf-3-1}, 
 we obtain related integrability conditions
 that ensure uniqueness of this solution.

 \medskip

 On the other hand, sufficient conditions for
  the integrability and uniform
  integrability of random weights
  over a time interval $[0,T]$ are provided
  in Theorems~\ref{integ-1}
  and \ref{unf-integ-1}
  under uniform bounds on the derivatives
  of equation coefficients, 
  provided that $T$ is sufficiently small.
  In Theorem~\ref{integ-3-2}
  and \ref{unf-integ-3-2}
  those conditions are then relaxed in order to
  allow for the growth of derivatives.
 Our results also include quantitative bounds of explosion times 
 that ensure the integrability of stochastic weights,
 the existence and uniqueness of solutions, and
 the validity of the stochastic representations of ODE solutions
 obtained in \cite{penent4}, \cite{huangprivault}. 

 \medskip
 
 Starting from the ODE~\eqref{ODE},
 our approach is to formulate the infinite ODE
 systems~\eqref{ode-system-a}-\eqref{ode-system-b}, 
 \eqref{ode-system-g}, whose solutions $\{x_g\}$
 are indexed by smooth functions $g$ in the set 
 $$
 {\cal C}: = \bigcup_{m\geq 1} {\cal C}^\infty(\R^d, \R^m).
 $$ 
 We then provide probabilistic representations for 
 $x_g (t)$, $g \in {\cal C}$,
 as well %
 as for 
 $g(x_\id (t))$, $g\in {\cal C}$,
 where $\id$ denotes the identity mapping on $\real^d$.
 This is achieved in Proposition~\ref{unf-intg}
 under uniform integrability conditions on functionals of
 random trees,
 which also shows the relation
$$
 x_g (t) = g(x_\id (t)), \qquad 
g\in
\bigcup_{m\geq 1} {\cal C}^\infty (\R^d, \R^m),
\quad t\geq 0. 
$$
 This representation is then used in Theorems~\ref{tjklf-3-2}
 and \ref{tjklf-3-1} to obtain the existence and uniqueness
 of ODE solutions represented as the expectation of a
 random tree functional.
 Those results rely on the integrability criteria established in 
 Theorems~\ref{integ-1},
 \ref{unf-integ-1}, \ref{integ-3-2} and \ref{unf-integ-3-2},
 based on the analysis 
 of the stability of the branching process
 that controls the growth of random trees. %
 This analysis is performed in 
 Proposition~\ref{radius-conv-weight-ct} 
 by controlling the integrability of
 marked random trees in which marks can be interpreted 
 as mutant types in population genetics models
 with mutation reversion (or back mutation).

\medskip

This paper is organized as follows. 
In Section~\ref{sect2-2} we introduce the basics of marked branching
processes and the related functionals of random trees
used for the
probabilistic representation \eqref{prob-rep-branching}. 
 In Section~\ref{sect3}
 we reformulate the ODE~\eqref{ODE} using
 ODE systems which are solved using probabilistic representations 
 under suitable integrability conditions of
 random weights.
 Those results are then applied in 
 Section~\ref{sect3-1} to
 the probabilistic representation
 of ODE solutions. 
In Section~\ref{sect5-1} %
 we derive
sufficient conditions for the integrability required 
in Section~\ref{sect3} via the analysis of related branching
processes.
In Sections~\ref{sect4} and \ref{sect4-1}
we provide sufficient 
conditions for the uniform integrability of
 random weights 
required in
Proposition~\ref{unf-intg}
 and Theorems~\ref{tjklf-3-2}-\ref{tjklf-3-1}. 
 Section~\ref{sect_ex} presents examples 
 of lifetime probability density function satisfying
 Assumption~\ref{asmpt-0}, together with examples of
 existence intervals. 
 In Appendix~\ref{btbt}, we present the complete construction of the
 rooted trees $\tau$ and their connection to
 the Butcher series \eqref{fjkld9}.    
      Numerical examples are available at 
 \url{https://github.com/nprivaul/mc-odes/blob/main/mc_odes.ipynb}.
   
\section{Marked branching trees} %
\label{sect2-2}
\noindent
 Our probabilistic representation of ODE solutions relies
 on a random branching tree $\mathcal{B}^{0,c}$ in which a label 
$$
 \mathbb K := \{\varnothing\}\cup \bigcup_{n\geq 1} \{1,2\}^n,
$$ 
 and a mark in the set of smooth functions 
\begin{equation*}
 {\cal C} := \bigcup_{m\geq 1} {\cal C}^\infty(\R^d, \R^m) 
\end{equation*}
 is attached to every branch,
 and the lifetimes of branches 
 are independent and identically distributed
 on a probability space $(\Omega , {\cal F}, \P)$ 
 according to a 
 common probability density function $\rho:\R^+\to\R$ 
 such that $\rho (t) > 0$ for all $t\in [0,T]$. 
\begin{definition}
 We let $\mathcal{B}^{0,c}$
 denote the random tree
 constructed as follows.
\begin{itemize}
\item %
  We start from a single branch with label $\varnothing$  
 and initial mark $c \in {\cal C}$.
 At the end of its lifetime $T_\varnothing := t_\varnothing$, 
 this branch yields either: 
 \begin{itemize}
\item
  a single offspring with label $(1)$ and mark $f$ if $c =\id$,
  \item
 two independent offsprings
 with respective labels $(1)$, $(2)$
 and mark $f$, $\nabla c$ if $c\ne\id$. 
\end{itemize}
 \item
  At generation $n \geq 1$, a branch having a parent label
  $\mathbf{k}^- := \left(k_1, \ldots, k_{n-1}\right)$ 
  starts at time $T_{{\mathbf k}^-}$
  and has the lifetime $t_{\mathbf k}$.
  At time $T_{\mathbf k} := T_{{\mathbf k}^-} + t_{\mathbf k}$, 
  it yields two independent %
  offsprings with the respective labels 
  $(\mathbf{k},1) = (k_1, \ldots , k_n,1)$, 
  $(\mathbf{k},2) = (k_1, \ldots , k_n,2)$
  and marks $f$, $\nabla c_{\mathbf k}$.
\end{itemize}
\end{definition}
\noindent 
\noindent
The set of labels of all branches living in the time interval $[0,t]$
is denoted by $\mathcal{K}_t$,
 its boundary at time $t$,
i.e. the set of labels of all branches living at time $t>0$, 
 is denoted by $\mathcal{K}^\pt_t$,
 and its interior in $[0,t)$,
   i.e. the set of branches that split not later than time $t$,
   is denoted by $\mathcal{K}^\circ_t$.
   Clearly, if the initial code is $c = \id$,
     then $|\mathcal K^\pt_t| = |\mathcal K^\circ_t|$; otherwise, $|\mathcal K^\pt_t| = |\mathcal K^\circ_t|+1$ and $B^{0,c}$ is a binary tree. 

   \medskip
    
 Figure~\ref{fjkld23-1} presents a sample 
 of $\mathcal{K}_t$ for the random tree $\mathcal{B}^{0,\id}$
 started from the mark $c=\id$,
 where $\id$ denotes the identity map on $\real^d$.
 
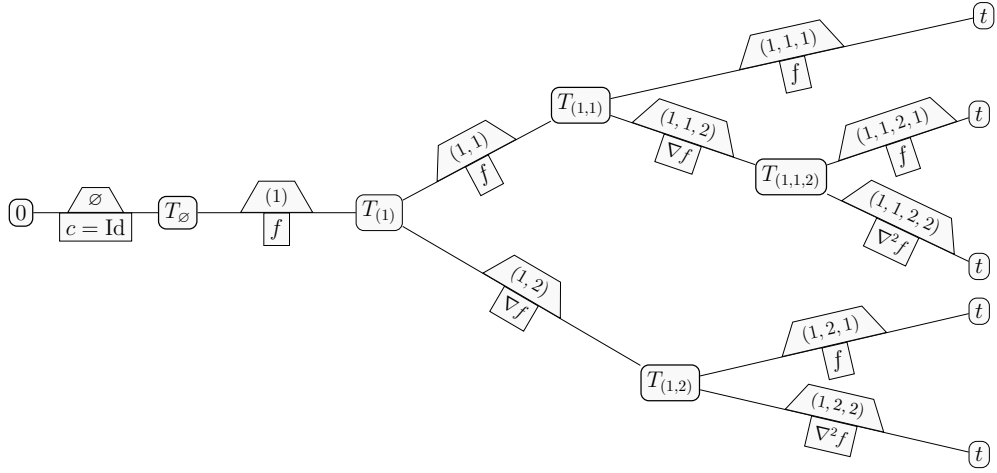
\begin{figure}[H]
\centering
\tikzstyle{level 1}=[level distance=4.5cm, sibling distance=4cm]
\tikzstyle{level 2}=[level distance=4.5cm, sibling distance=5cm]
\tikzstyle{level 2}=[level distance=4.5cm, sibling distance=4cm]
\begin{center}
  \hspace{-1.36cm}
  \resizebox{0.8\textwidth}{!}{
  \begin{tikzpicture}[scale=1.0,grow=right, sloped]
\node[rectangle,draw,black,text=black,thick, rounded corners=5pt,fill=gray!05]{\large $0$} 
child {
  node[%
     draw,black,text=black,thick, rounded corners=5pt,xshift=-1cm,fill=gray!05] {\large $T_\varnothing$} %
            child {
  node[name=data, %
         draw, rounded corners=5pt,fill=gray!05] (main) 
        {\large $T_{(1)}$} %
                child{
                  node[%
                     draw,black,text=black,thick,yshift=-1.8cm,xshift=2cm,rounded corners=5pt,fill=gray!05]{\large $T_{(1,2)}$} %
                    child{
    node[name=data, %
         draw, rounded corners=5pt,yshift=0.4cm,fill=gray!05,xshift=2.4cm]  
        {\large $t$} %
                    edge from parent
                    node[above,trapezium,draw,fill=gray!05]{\small $(1,2,2)$}
                    node[rectangle,draw,below,fill=gray!05]{$\nabla^2 f$}%
                    }
                    child{
    node[name=data, %
         draw, rounded corners=5pt,yshift=-0.4cm,xshift=2.4cm,fill=gray!05]  
        {\large $t$ \nodepart{two} $f$}
                    edge from parent
                    node[above,trapezium,draw,fill=gray!05]{$(1,2,1)$}
                    node[rectangle,draw,below,fill=gray!05]{\large $f$} %
                    }
                edge from parent
                node[above,trapezium,draw,fill=gray!05]{$(1,2)$}
                node[rectangle,draw,below,fill=gray!05]{$\nabla f$}%
                }
                child{
                  node[%
                     draw,black,text=black,thick,yshift=0.4cm,rounded corners=5pt,fill=gray!05]{\large $T_{(1,1)}$} %
                    child{
                      node[%
                        draw,black,text=black,thick,yshift=0.4cm,xshift=0.2cm,rounded corners=5pt,fill=gray!05]{\large $T_{(1,1,2)}$} %
                    child{
    node[name=data, %
         draw, rounded corners=5pt,yshift=0cm,fill=gray!05,xshift=-0.3cm]  
        {\large $t$} %
                    edge from parent
                    node[above,trapezium,draw,fill=gray!05]{$(1,1,2,2)$}
                    node[rectangle,draw,below,fill=gray!05]{$\nabla^2 f$}%
                    }
                    child{
    node[name=data, %
         draw, rounded corners=5pt,yshift=-0.6cm,fill=gray!05,xshift=-0.3cm]  
        {\large $t$} %
                    edge from parent
                    node[above,trapezium,draw,fill=gray!05]{$(1,1,2,1)$}
                    node[rectangle,draw,below,fill=gray!05]{\large $f$} %
                    }
                    edge from parent
                    node[above,trapezium,draw,fill=gray!05]{$(1,1,2)$}
                    node[rectangle,draw,below,fill=gray!05]{$\nabla f$} %
                    }
                    child{
    node[name=data, %
         draw, rounded corners=5pt,fill=gray!05,xshift=4.5cm]  
        {\large $t$ \nodepart{two} $f$}
                    edge from parent
                    node[above,trapezium,draw,fill=gray!05]{$(1,1,1)$}
                    node[rectangle,draw,below,fill=gray!05]{\large $f$} %
                    }
                edge from parent
                node[above,trapezium,draw,fill=gray!05]{$(1,1)$}
                node[rectangle,draw,below,fill=gray!05]{\large $f$} %
                }
                edge from parent
                node[above,trapezium,draw,fill=gray!05] {$(1)$}
                node[rectangle,draw,below,fill=gray!05] {\large $f$} %
            }
                edge from parent
                node[above,trapezium,draw,fill=gray!05] {$\varnothing$}
                node[rectangle,draw,fill=gray!05,below] {\large $c={\rm Id}$} %
};
\end{tikzpicture}
}
\end{center}
\vskip-0.3cm
\caption{Sample of
  $\mathcal{K}_t$ for the random tree
  ${\mathcal{B}}^{0,c}$ started from $c=\id$.}
\label{fjkld23-1}
\end{figure}

\vspace{-.3cm}

\noindent
 Figure \ref{fjkld23-2} presents a sample 
 of $\mathcal{K}_t$ for the random tree
 $\mathcal{B}^{0,c}$ started from a mark $c\not= {\rm Id}$.
 
\begin{figure}[H]
\centering
\tikzstyle{level 1}=[level distance=4.5cm, sibling distance=4cm]
\tikzstyle{level 2}=[level distance=4.5cm, sibling distance=5cm]
\tikzstyle{level 2}=[level distance=4.5cm, sibling distance=4cm]
\begin{center}
  \hspace{-1.44cm}
  \resizebox{0.75\textwidth}{!}{
\begin{tikzpicture}[scale=1.0,grow=right, sloped]
        \node[draw,black,text=black,thick, rounded corners=5pt,fill=gray!05,] {\large $0$} 
            child {node[name=data, fill=gray!05,%
         draw, rounded corners=5pt] (main) 
        {\large $T_\varnothing$} %
                                child{
                  node[%
                    draw,fill=gray!05,text=black,thick,yshift=0.4cm,rounded corners=5pt,xshift=1cm]{\large $T_{(1)}$ \nodepart{two} $\nabla c$} 
                    child{
    node[name=data, %
         draw, fill=gray!05,rounded corners=5pt,yshift=0.4cm,xshift=4cm]  
        {\large $t$ \nodepart{two} $\nabla^2 c$}
                    edge from parent
                    node[above,trapezium,draw,fill=gray!05]{$(1,2)$}
                    node[rectangle,draw,fill=gray!05,below]{$\nabla^2 c$}%
                    }
                    child{
    node[name=data, %
         draw, fill=gray!05, rounded corners=5pt,yshift=-0.3cm,xshift=4cm]  
        {\large $t$ \nodepart{two} $f$}
                    edge from parent
                    node[above,trapezium,draw,fill=gray!05]{$(1,1)$}
                    node[rectangle,draw,fill=gray!05,below]{\large $f$} %
                    }
                edge from parent
                node[above,trapezium,draw,fill=gray!05]{$(1)$}
                node[rectangle,draw,fill=gray!05,below]{$\nabla c$}%
                }
                                child{
                  node[%
                     draw,black,text=black,thick,yshift=2.8cm,fill=gray!05,rounded corners=5pt]{\large $T_{(2)}$ \nodepart{two} $f$} %
                    child{
                      node[%
                         draw,black,text=black,thick,yshift=0.4cm,fill=gray!05,rounded corners=5pt]{\large $T_{(2,2)}$ \nodepart{two} $\nabla f$} 
                    child{
    node[name=data, %
         draw,fill=gray!05, rounded corners=5pt,yshift=0cm,xshift=0.5cm]  
        {\large $t$ \nodepart{two} $\nabla^2 f$}
                    edge from parent
                    node[above,trapezium,draw,fill=gray!05]{$(2,2,2)$}
                    node[rectangle,draw,fill=gray!05,below]{$\nabla^2 f$}%
                    }
                    child{
    node[name=data, %
         draw,fill=gray!05, rounded corners=5pt,yshift=-0.6cm,xshift=0.5cm]  
        {\large $t$ \nodepart{two} $f$}
                    edge from parent
                    node[above,trapezium,draw,fill=gray!05]{$(2,2,1)$}
                    node[rectangle,draw,fill=gray!05,below]{\large $f$} %
                    }
                    edge from parent
                    node[above,trapezium,draw,fill=gray!05]{$(2,2)$}
                    node[rectangle,draw,fill=gray!05,below]{$\nabla f$} %
                    }
                    child{
    node[name=data, %
         draw, fill=gray!05, rounded corners=5pt,xshift=5cm,yshift=-0.5cm]  
        {\large $t$ \nodepart{two} $f$}
                    edge from parent
                    node[above,trapezium,draw,fill=gray!05]{$(2,1)$}
                    node[rectangle,draw,fill=gray!05,below]{\large $f$} %
                    }
                edge from parent
                node[above,trapezium,draw,fill=gray!05]{$(2)$}
                node[rectangle,draw,fill=gray!05,below]{\large $f$} %
                }
                edge from parent
                node[above,trapezium,draw,fill=gray!05] {$\varnothing$}
                node[rectangle,draw,fill=gray!05,below]  {\large $c\not={\rm Id}$} %
}; %
\end{tikzpicture}
}
\end{center}
\vskip-0.3cm
\caption{Sample of
   $\mathcal{K}_t$ for the random tree
  ${\mathcal{B}}^{0,c}$ started from $c\not=\id$.}
\label{fjkld23-2}
\end{figure}

\vskip-0.3cm

\noindent
 For $n\geq 1$, we also denote by $\mathcal{K}^n_t$ the label set of
 the branches generated at the $n$-$th$ generation, and let
   $$
   \mathcal{K}^{\circ,n}_t := \mathcal{K}^\circ_t \cap \mathcal{K}^n_t,
   \quad
   \mathcal{K}^{\pt,n}_t := \mathcal{K}^\pt_t \cap \mathcal{K}^n_t.
$$
  \noindent 
   When the initial mark is $c = g
   \in {\cal C} \setminus \{ \id \}$, we observe that the marks
   of all branches are of the form either $\nabla^m f$ or $\nabla^m g$, $m\geq 0$.
 Letting  
 $$
 {\cal C}_g := \{ \nabla^k g \}_{k\geq 0},
 \quad
 g\in {\cal C},
 $$ 
 we also consider the classes of $n$-$th$ generation branches 
\begin{equation*}
  \mathcal{K}^n_t(f) := \left\{ \mathbf k \in \mathcal{K}^n_t \ : \ c_{\mathbf k} \in {\cal C}_f \right\} \quad \mbox{and} \quad \mathcal{K}^n_t(g) := \left\{ \mathbf k \in \mathcal{K}^n_t \ : \ c_{\mathbf k} \in {\cal C}_g \right\}. 
\end{equation*}
 Finally, we denote by 
$$
 \widebar{F}_\rho(t) := \P\left( T_\varnothing>t\right) = \int_t^\infty \rho (r) dr 
 , \qquad t \geq 0, 
$$
 the tail cumulative distribution function of $\rho$,
 and we consider the following functional
 of random trees. 
\begin{definition}
   We let $\mathcal{H}_t \left(\mathcal{B}^{0,c}\right)$ 
denote the functional of $\mathcal{B}^{0,c}$ defined as    
      \begin{equation}
        \label{functional}
        \mathcal{H}_t \left(\mathcal{B}^{0,c}\right):=
        \prod_{\mathbf k \in \mathcal{K}^{\partial}_t} \frac{c_{\mathbf k}(x_0)}{\widebar{F}_\rho (t-T_{{\mathbf k}^-})} \prod_{\mathbf k \in \mathcal{K}^\circ_t} \frac{1}{\rho(t_\mathbf k)},
   \quad   t\in [0,T], \quad c \in {\cal C}_f \cup {\cal C}_g. 
\end{equation}
\end{definition}
\noindent
 In \eqref{functional}, the terms in the right product make sense 
 due to the assumption that $\rho(t) > 0$ for all $t\in [0,T]$.
 In addition, the product of $c_{\mathbf k}(x_0)$ over 
 $k \in \mathcal{K}^{\partial}_t$
 is interpreted as a composition
 according to tree leaves from bottom to top 
 of the elementary differentials
 $f$, $\nabla f(f)$, $\nabla^2 f(f,f)$, $\nabla f(\nabla f (f))$
 defined as 
\begin{align}
\label{aa1}
  & 
( \nabla f )(f) (x_0) := \left(  \sum_{i_2=1}^d \frac{\partial f_{i_1}}{\partial x_{i_2}} (x_0) f_{i_2} (x_0) \right)_{i_1=1,\ldots , d},
\\
\nonumber
\\
\nonumber %
  & \nabla^2 f ( f,f ) (x_0)  
:= \left( \sum_{i_2,i_3=1}^d \frac{\partial^2 f_{i_1}}{\partial x_{i_2}\partial x_{i_3}} (x_0) 
f_{i_2} (x_0) f_{i_3} (x_0) \right)_{i_1=1,\ldots , d},
\\
\nonumber
\\
\label{aa2}
  & \nabla f ( \nabla f (f )) (x_0)  
:= \left( \sum_{i_2,i_3=1}^d \frac{\partial f_{i_1}}{\partial x_{i_2}} (x_0) 
 \frac{\partial f_{i_2}}{\partial x_{i_3}} (x_0) f_{i_3} (x_0) 
 \right)_{i_1=1,\ldots , d},
\end{align} 
etc. 
 For example, in the sample tree of Figure~\ref{fjkld23-1}
 we have 
 $$
 \mathcal{K}^\circ_t = \{(1,1,1),(1,1,2,1),(1,1,2,2),(1,2,1),(1,2,2)\}$$
 and
 $$
 \prod_{\mathbf k \in \mathcal{K}^\partial_t}
 c_{\mathbf k}(x_0)
 =
 \nabla^2 f (f ,\nabla^2 f (f,f)) (x_0) ,
 $$
 while in the sample tree of Figure~\ref{fjkld23-2}
 we have
 $$
 \mathcal{K}^\circ_t = \{(2,2),(2,1),(1,2,2),(1,2,1),(1,1)\}
 $$
  and 
$$
 \prod_{\mathbf k \in \mathcal{K}^\partial_t} c_{\mathbf k}(x_0) 
 =
 ( \nabla^2 c) (f , (\nabla^2 f) (f,f)) (x_0),
$$
 see Appendix~\ref{btbt} for details.  
\section{ODE systems} 
\label{sect3}

\noindent 
 In this section, we reformulate the ODE~\eqref{ODE}
 as an infinite system of differential equations
 indexed by the set ${\cal C}$ of marks.
\begin{lemma}\label{relation-ODE-ODEs}
  Let $T>0$ such that the ODE~\eqref{ODE} admits
  a solution $x\in {\cal C}^1([0,T],\R^d)$.
  Then, the family
  $(x_c)_{c\in \{\id\}\cup \left\{ \nabla^k f \right\}_{k\geq 0}}
  : =
  (c(x))_{c\in \{\id\}\cup \left\{ \nabla^k f \right\}_{k\geq 0}}$, %
 solves the ODE system 
   \begin{subequations}
        	\begin{empheq}[left=\empheqlbrace]{align}
        	\label{ode-system-a}
                &
                 (x_c )' = x_f, \quad c = {\rm Id},  
 \medskip
    \\
        	\label{ode-system-b}
                &
              ( x_c )' = 
                x_f x_{\nabla c}, \quad c = \nabla^m f, \ m \geq 0, 
\end{empheq}
\end{subequations}
   with the initial conditions $x_c(0)=c(x_0)$,
   for $c \in \{\id\}\cup \left\{ \nabla^k f \right\}_{k\geq 0}$. 
\end{lemma}
\begin{Proof}
 It follows from the time differentiation 
 \begin{equation*}
  g(x(t))' = x(t) \nabla g (x(t)) 
  = f(x(t)) \nabla g (x(t)) , 
 \quad t\in [0,T], 
\end{equation*}
 that for every test function $g \in {\cal C}$,
 the family 
  $( g(x) )_{g\in{\cal C}}$
  satisfies the system of equations 
  \begin{equation}
\nonumber %
    (g(x))' = \left\{
    \begin{array}{ll}
    f(x), & g = \id, 
    \smallskip
    \\ 
    f(x) \nabla g(x), &  g \in {\cal C} \setminus \{ \id \}, 
    \end{array}
    \right. 
    \end{equation}
  with initial condition $(x_g(x)(0))_{g\in {\cal C}} = (g(x_0))_{g\in {\cal C}}$. 
\end{Proof}
\noindent
\begin{prop}
    \label{Picard}
 Assume that %
  \begin{enumerate}[\hypertarget{asmp-initial-lp}{\bf (C1)}]
    \item 
 $(\nabla^k f(x_0))_{k\geq 0} \in \ell^p(\inte )$ for some $p\in [1,\infty ]$.
\end{enumerate}
 Then, the subsystem~\eqref{ode-system-b} 
 admits
 a unique local solution %
 in $\ell^p(\inte )$. 
\end{prop}
\begin{Proof}
 Using the left shift operator
 that acts on real sequences as
 $$
  S(x^{(0)}, x^{(1)}, \ldots) = (x^{(1)}, x^{(2)}, \ldots), 
$$
 we rewrite
  the subsystem~\eqref{ode-system-b}, 
  as the infinite-dimensional ODE 
  $$
  {\mathbf x}' = x^{(0)} \cdot S(\mathbf x).
  $$ 
 Since the nonlinearity $F(\mathbf x) := x^{(0)} \cdot S(\mathbf x)$ 
 is continuous from $\ell^p(\inte )$ into itself %
 together with its (non-vanishing) Fr\'echet derivatives %
  \begin{equation*}
    DF(\mathbf x)(\mathbf u) = u^{(0)} \cdot S(\mathbf x) + x^{(0)} \cdot S(\mathbf u)
    \quad
    \mbox{and}
  \quad 
    D^2F(\mathbf x)(\mathbf u, \mathbf v) = u^{(0)} \cdot S(\mathbf v) + v^{(0)} \cdot S(\mathbf u),
  \end{equation*}
  local existence 
  follows from the infinite-dimensional version of the Picard--Lindel\"of theorem, 
  cf. e.g. \cite[Theorem 25]{LS94},
  and shows the well-posedness of \eqref{ode-system-b}. 
\end{Proof} 
\noindent
The well-posedness of the ODE system~\eqref{ode-system-a}-\eqref{ode-system-b} 
 now follows from that of \eqref{ode-system-b}. 
  We also note that the local well-posedness of \eqref{ode-system-a}-\eqref{ode-system-b} only relies on the regularity of $f$ at the point $x_0$, %
  whereas that of the ODE~\eqref{ODE} relies on the local regularity of $f$ near $x_0$.
\noindent
 In what follows, given $g \in {\cal C} \setminus \{ \id \}$ we 
 augment the subsystem~\eqref{ode-system-b}
 as %
\begin{equation}
 \label{ode-system-g}
 \left\{
  \begin{array}{l} 
    ( x_c )' = x_f x_{\nabla c},
    \quad c = \nabla^m g, 
    \medskip
    \\
    (x_c)' = x_f x_{\nabla c},
    \quad c = \nabla^m f, 
    \quad     m \geq 0,
\end{array}
  \right. 
\end{equation}
 with the initial conditions 
\begin{equation}
\nonumber %
  \left\{
  \begin{array}{l} 
 x_c (0) = \nabla^m g(x_0), \quad c = \nabla^m g, 
    \medskip
    \\
 x_c(0) = \nabla^m f(x_0), \quad c = \nabla^m f,  
 \quad m\geq 0.
  \end{array}
  \right. 
\end{equation}
\begin{lemma}
\label{lemma-uc}
 Assume that
 $$
 \E \big[ \big| \mathcal{H}_t \big(\mathcal{B}^{0,c}\big)\big| \big] < \infty,
      \quad \mbox{for all } t\in [0,T] \mbox{ and } c \in {\cal C}. %
      $$
 Then, the functions defined by 
\begin{equation}\label{uc}
    x_c (t) := \E \big[ \mathcal{H}_t\big(\mathcal{B}^{0,c}\big) \big],
  \quad t\in [0,T], \quad c \in {\cal C}, %
  \end{equation}
 satisfy the recursion 
\begin{equation*}%
  x_c (t) = c (x_0) + \int_0^t \left( x_f(s) \ind_{\{c=\id\}} + x_{\nabla c}(s) x_f(s) \ind_{\{c\ne\id\}} \right) ds,
  \quad t\in[0,T],
    \
    c \in {\cal C}. %
\end{equation*}
In particular,
\begin{enumerate}[i)]
  \item 
      The family $(x_c)_{c\in {\cal C}_f\cup{\cal C}_g}$
      solves the ODE system~\eqref{ode-system-g}, i.e.
$$
 \left\{
  \begin{array}{l} 
    (x_{\nabla^m g})' = x_f x_{\nabla^{m+1} g}, 
    \medskip
    \\
    (x_{\nabla^m f})' = x_f x_{\nabla^{m+1} f},
    \quad     m \geq 0. 
\end{array}
  \right. 
$$ 
          \item 
            If $x_f = f(x_{\id})$, then $x_{\id}$ is a solution of
            the ODE~\eqref{ODE}.
\end{enumerate}
\end{lemma}
\begin{proof}
 Given $c \in {\cal C} \setminus \{ \id \}$, 
 we observe that $\mathcal{H}_t \left(\mathcal{B}^{0,c}\right)$
 satisfies %
\begin{align}
\nonumber %
  \mathcal{H}_t \left(\mathcal{B}^{0,c}\right) & = \mathcal{H}_t\left(\mathcal{B}^{0,c}\right) \left( \ind_{\{T_\varnothing>t\}} + \ind_{\{T_\varnothing\leq t\}} \right)
  \\
  \nonumber
  & = \frac{c (x_0)}{\widebar{F}_\rho (t-T_{\varnothing^{-}})} \ind_{\{T_\varnothing>t\}} + \ind_{\{T_\varnothing\leq t\}} \frac{1}{\rho(t_\varnothing)} \prod_{\mathbf k \in \mathcal{K}_t^{\partial}} \frac{c_{\mathbf k}(x_0)}{\widebar{F}_\rho (t-T_{\mathbf k^{-}})} \prod_{\mathbf k \in \mathcal{K}_t^\circ \setminus \{\varnothing\}} \frac{1}{\rho(t_\mathbf k)}
  \\
  \nonumber
  & = \frac{c (x_0)}{\widebar{F}_\rho(t)} \ind_{\{T_\varnothing>t\}} + \ind_{\{T_\varnothing\leq t\}} \frac{1}{\rho(t_\varnothing)}
  \big( \mathcal{H}_t\left(\mathcal{B}^{T_\varnothing,f}\right) \ind_{\{c=\id\}} + \mathcal{H}_t\left(\mathcal{B}^{T_\varnothing,\nabla c}\right) \mathcal{H}_t\left(\mathcal{B}^{T_\varnothing,f}\right) \ind_{\{c\ne\id\}} \big). 
\end{align} 
 Hence, by independence of tree branches, we have 
\begin{align*}
  x_c (t) &= \E \left[ \mathcal{H}_t\left(\mathcal{B}^{0,c}\right) \right]
  \\
  &= \E \left[ \frac{c (x_0)}{\widebar{F}_\rho(t)} \ind_{\{t_\varnothing>t\}} + \ind_{\{t_\varnothing\leq t\}} \frac{1}{\rho(t_\varnothing)} \mathcal{H}_t\left(\mathcal{B}^{t_\varnothing,\nabla c}\right) \mathcal{H}_t\left(\mathcal{B}^{t_\varnothing,f}\right) \right]
  \\
  &= c (x_0) \E \left[ \frac{1}{\widebar{F}_\rho(t)} \ind_{\{t_\varnothing>t\}} \right]
  + \E \left[ \ind_{\{t_\varnothing\leq t\}} \frac{1}{\rho(t_\varnothing)}
    \E \left[
      \mathcal{H}_t\left(\mathcal{B}^{t_\varnothing,\nabla c}\right)
      \ \! \big| \ \!
      t_\varnothing
      \right]
    \E \left[   \mathcal{H}_t\left(\mathcal{B}^{t_\varnothing,f}\right)
            \ \! \big| \ \!
      t_\varnothing
      \right]
\right]
  \\
  &= c (x_0) \E \left[ \frac{1}{\widebar{F}_\rho(t)} \ind_{\{t_\varnothing>t\}} \right]
  + \E \left[ \ind_{\{t_\varnothing\leq t\}} \frac{1}{\rho(t_\varnothing)}
    x_{\nabla c}(t - t_\varnothing )
    x_f(t - t_\varnothing ) 
\right]
\\
&= c (x_0) + \int_0^t x_{\nabla c}(t-s) x_f(t-s) ds
\\
    &= c (x_0) + \int_0^t x_{\nabla c}(s) x_f(s) ds.
\end{align*}
 In case $c=\id$, we similarly have
\begin{align*}
  x_\id (t) &= \E \left[ \mathcal{H}_t\left(\mathcal{B}^{0,\id}\right) \right]
  \\
  &= \E \left[ \frac{x_0}{\widebar{F}_\rho(t)} \ind_{\{t_\varnothing>t\}} + \ind_{\{t_\varnothing\leq t\}} \frac{1}{\rho(t_\varnothing)}
    \mathcal{H}_t\left(\mathcal{B}^{t_\varnothing,f}\right) \right]
  \\
    &= x_0 + 
    \int_0^t
    \E \left[
      \mathcal{H}_t\left(\mathcal{B}^{s,f}\right)
      \right]
    ds
          \\
          &= x_0 + \int_0^t x_f(t-s) ds
                   \\
    &= x_0 + \int_0^t x_f(s) ds.
\end{align*}
\end{proof}
\noindent 
 As a consequence of
 Proposition~\ref{Picard} and Lemma~\ref{lemma-uc} we have the following proposition, which recovers the conclusion of Theorem~4.2 in \cite{penent4}. 
\begin{prop}
 Under the conditions of Lemma~\ref{lemma-uc},
 assume that
 $(\nabla^k f)_{k\geq 0}$ 
 satisfies 
 \begin{enumerate}[\hypertarget{asmp-lp}{\bf (C2)}]
    \item 
    $(\nabla^k f(y))_{k\geq 0} \in \ell^p(\inte )$
      for some $p \in [1, \infty ]$  and
   all $y$ in a neighborhood $U$ of $x_0$
    in $\R^d$.     
  \end{enumerate}
 Then, for some $T>0$, $(x_\id(t))_{t\in [0,T]}$ defined
 in \eqref{uc} is a solution of the ODE~\eqref{ODE}. 
\end{prop}
\begin{Proof} 
 Under Condition~\hyperlink{asmp-lp}{(C2)} 
 the ODE~\eqref{ODE} admits a ${\cal C}^1$ solution
 such that for some $T>0$ we have 
$$
 ( \nabla^k f(x(t)))_{k\geq 0} \in \ell^p(\inte ), \qquad t\in [0,T], 
$$
 for some $p \in [1,\infty ]$.
 In addition, from Proposition~\ref{Picard}
 the subsystem~\eqref{ode-system-b}
 admits a unique solution 
 in $\ell^p(\inte )$ on $[0,T]$,
 hence by Lemma~\ref{lemma-uc}
 we have 
 $x=x_{\id}$ and
 $\nabla^m f(x) = x_{\nabla^m f}$ for all $m\geq 0$.
\end{Proof} 
\noindent
 In comparison with Condition~\hyperlink{asmp-initial-lp}{(C1)}, %
  Condition~\hyperlink{asmp-lp}{(C2)} %
  fills the gap of regularity requirement on $f$
  between the well-posedness of
  the ODE system~\eqref{ode-system-a}-\eqref{ode-system-b}
 and that of the ODE~\eqref{ODE}, 
 as observed in Proposition~\ref{Picard}.

 \medskip

 Proposition~\ref{unf-intg} provides a probabilistic
interpretation for the solutions of
 the ODE system~\eqref{ode-system-a}-\eqref{ode-system-b}. 
 Sufficient conditions on
 $f$ for the integrability requirements on \eqref{functional-2} and
 \eqref{functional-1} below are provided in  
 Theorems~\ref{integ-1} and \ref{unf-integ-1}. 
 { \begin{prop}
  \label{unf-intg}
  Let $g \in {\cal C} \setminus \{ \id \}$. 
  \begin{enumerate}[i)]
  \item
    \label{unf-intg-2}
  Assume that the ODE system~\eqref{ode-system-a}-\eqref{ode-system-b} admits a
  solution
  $$
  \{ x_c \}_{c\in \{\it \id\}\cup {\cal C}_f }
  \subset \bigcup_{m\geq 1} {\cal C}^1([0,T],\R^m),
  $$
  and that the sequence of functionals of the tree
  $\mathcal{B}^{0,g}$ defined as 
  \begin{equation}
    \label{functional-2}
    \widetilde{\mathcal{H}}_t^{g,n}
    \big(
    \mathcal{B}^{0,g}
    \big)
    :=
    \hskip-0.1cm
    \prod_{\mathbf k \in \cup_{i=0}^n \mathcal{K}^{\partial,i}_t}
    \hskip-0.1cm
    \frac{c_{\mathbf k}(x_0) }{\widebar{F}_\rho (t-T_{\mathbf k^{-}})}
    \prod_{\mathbf k \in \cup_{i=0}^n \mathcal{K}^{\circ,i}_t}
    \frac{1}{\rho(t_{\mathbf k})}
    \prod_{\mathbf k \in \mathcal{K}^{n+1}_t(f)}
    \hskip-0.1cm
    x_{c_{\mathbf k}} \left( T_{\mathbf k^{-}} \right)
    \hskip-0.1cm
    \prod_{\mathbf k \in \mathcal{K}^{n+1}_t(g)}
    \hskip-0.1cm
    c_{\mathbf k} \left( x_\id(T_{\mathbf k^{-}}) \right)
 \end{equation}
  is uniformly integrable in $n\geq 0$ for all $t\in [0,T]$. Then, 
  the following representation holds: 
  \begin{equation}
    \label{fjkl1} 
  g(x_\id(t)) = \E \left[ \mathcal{H}_t\left(\mathcal{B}^{0,g}\right) \right],
  \qquad t\in[0,T].
\end{equation}
\item
  \label{unf-intg-i}
    Assume that the ODE system~\eqref{ode-system-g} admits a solution
  $$
  \{ x_c \}_{c\in{\cal C}_f\cup {\cal C}_g }
  \subset \bigcup_{m\geq 1} {\cal C}^1([0,T],\R^m), 
  $$
  and that the sequence 
  \begin{equation}
    \label{functional-1}
    \mathcal{H}_t^{g,n}
        \big(
    \mathcal{B}^{0,g}
    \big)
 := \prod_{\mathbf k \in \cup_{i=0}^n \mathcal{K}^{\partial,i}_t} \frac{c_{\mathbf k}(x_0) }{\widebar{F}_\rho (t-T_{\mathbf k^{-}})} \prod_{\mathbf k \in \cup_{i=0}^n \mathcal{K}^{\circ,i}_t} \frac{1}{\rho(t_{\mathbf k})} \prod_{\mathbf k \in \mathcal{K}^{n+1}_t} x_{c_{\mathbf k}} \left( T_{\mathbf k^{-}} \right)  
\end{equation}
  is uniformly integrable in $n\geq 0$ for all $t\in [0,T]$. 
  Then, the solution of the ODE system~\eqref{ode-system-g}
  is unique and the following representation holds: 
  \begin{equation}
    \label{fjkdls13} 
  x_g(t) = \E \left[ \mathcal{H}_t\left(\mathcal{B}^{0,g}\right) \right],
  \qquad t\in[0,T]. 
\end{equation}
\end{enumerate} 
  Moreover, if all above conditions are satisfied, then
  for all $g \in {\cal C} \setminus \{ \id \}$ we have
  $$
  g(x_\id(t)) =
  x_g(t) = \E \left[ \mathcal{H}_t\left(\mathcal{B}^{0,g}\right) \right],
  \qquad t\in[0,T].
$$ 
\end{prop}
\begin{proof}
$i)$ It follows from \eqref{ode-system-a}-\eqref{ode-system-b} that
\begin{align}
  \label{eqn-1-1}
  g(x_\id(t)) &= g (x_0) + \int_0^t \nabla g(x_\id(s)) x_f(s) ds
  \\
  \nonumber
  &= \E \left[ \frac{g (x_0)}{\widebar{F}_\rho(t)} \ind_{\{ T_\varnothing>t \}} + \ind_{\{T_\varnothing\leq t \}}\frac{1}{\rho(t_\varnothing)} \nabla g (x_\id(T_\varnothing )) x_f(T_\varnothing ) \right]
\\
  \nonumber
  &= 
 \E\big[\widetilde{\mathcal{H}}_t^{g,0}\big], \quad t\in [0,T].
\end{align} 
 Next, %
repeating the argument leading to \eqref{eqn-1-1}, 
we expand $\nabla g (x_\id(T_\varnothing ))$ for $\mathbf{k}=(1)$ as 
\begin{align}
  \nonumber
  \nabla g (x_\id(T_\varnothing )) & = \nabla g (x_\id( T_{\mathbf{k}^-} ))
  \\
  \label{fjlkf3}
  & 
  = \E \left[ \frac{\nabla g (x_0)}{\widebar{F}_\rho (t-T_{\mathbf{k}^-})} \ind_{\{ T_\mathbf{k}>t \}} + \ind_{\{T_\mathbf{k}\leq t \}} \frac{1}{\rho(t_\mathbf{k})} \nabla^2 g (x_\id(T_\mathbf{k} )) x_f(T_\mathbf{k} ) \ \! \Bigg| \ \! \F_{T_\varnothing} \right], 
\end{align}
 and plugging \eqref{fjlkf3} back into \eqref{eqn-1-1}
 shows that $g(x_\id(t)) = \E \big[\widetilde{\mathcal{H}}_t^{g,1}\big]$
 by conditional independence of branches in $\mathcal{K}^1_t$ given $\F_{T_\varnothing}$.
 Similarly, we can show by iterations that 
$$
g(x_\id(t)) = \E \big[\widetilde{\mathcal{H}}_t^{g,n}\big]
$$ 
 for all $n\geq 0$. 
 Then \eqref{fjkl1} follows by taking the limit as $n$ tends to infinity,
 from
 the uniform integrability of $\big(\widetilde{\mathcal{H}}_t^{g,n}\big)_{n\geq 0}$
 and the almost sure convergence of 
 $\widetilde{\mathcal{H}}_t^{g,n}
 $ to $\mathcal{H}_t(\mathcal{B}^{0,g})$. 

 \smallskip
 
\noindent
$ii)$ 
 By \eqref{ode-system-g}, we have 
\begin{align}
  \label{eqn-1}
  x_g(t) &= x_g (0) + \int_0^t x_{\nabla g}(s) x_f(s) ds
  \\
  \nonumber
  &= \E \left[ \frac{g (x_0)}{\widebar{F}_\rho(t)} \ind_{\{ T_\varnothing>t \}} + \ind_{\{T_\varnothing\leq t \}}\frac{1}{\rho(t_\varnothing)} x_{\nabla g}(T_\varnothing ) x_f(T_\varnothing ) \right]
  \\
  \nonumber
  & = \E[\mathcal{H}_t^{g,0}], \quad t\in [0,T].
\end{align}
 Next, since each offspring has same dynamics as
its parent branch, we can repeat the above argument to the 
 branch $\mathbf{k} = (1) \in \mathcal{K}^1_t$ with mark 
$g_1 \in \{ \nabla g , f \}$, 
to get 
\begin{align}
  \nonumber
  x_{g_1}(T_\varnothing ) & = x_{g_1}\left( T_{\mathbf{k}^-} \right)
  \\
  \label{eqn-10}
  & = \E \left[ \frac{g_1 (x_0)}{\widebar{F}_\rho (t-T_{\mathbf{k}^-})} \ind_{\{ T_\mathbf{k}>t \}} + \ind_{\{T_\mathbf{k}\leq t \}} \frac{1}{\rho(t_\mathbf{k})} x_{\nabla g_1}(T_\mathbf{k} ) x_f(T_\mathbf{k} ) \ \!
    \Bigg|  \ \!
    \F_{T_\varnothing} \right], 
\end{align} 
 and plugging \eqref{eqn-10} into \eqref{eqn-1}
 yields $x_g(t) = \E[\mathcal{H}_t^{g,1}]$ by conditional independence of
 branches in $\mathcal{K}^1_t$ given $\F_{T_\varnothing}$.
 Similarly, by iterations we find 
 $$x_g(t) = \E[\mathcal{H}_t^{g,n}],
 \qquad n\geq 0.
 $$ 
 As $\mathcal{H}_t^{g,n}$ converges to $\mathcal{H}_t(\mathcal{B}^{0,g})$ almost surely
 as $n$ tends to infinity, we obtain \eqref{fjkdls13} 
 from the uniform integrability of $\big(\mathcal{H}_t^{g,n}\big)_{n\geq 0}$. 
\end{proof}
\noindent 
\section{Probabilistic solution of ODEs} 
\label{sect3-1}
\noindent 
 Theorems~\ref{tjklf-3-2} and \ref{tjklf-3-1}
 show the existence and uniqueness of 
 a solution to the ODE~\eqref{ODE} on a time interval $[0,T]$ 
 by a probabilistic representation argument, 
 under uniform integrability assumptions on the sequences 
 $
  \big(  \widetilde{\mathcal{H}}_t^{f,n}
    \big(
    \mathcal{B}^{0,f}
    \big)
    \big)_{n\geq 0}
    $ and
    $\big( \mathcal{H}_t^{g,n}
        \big(
    \mathcal{B}^{0,g}
    \big)\big)_{n\geq 0}$
        respectively defined in \eqref{functional-2} 
 and 
 \eqref{functional-1}.
 In comparison with Theorem~4.2 of
 \cite{penent4}, the next result does
 not assume the existence
 of a solution to the ODE~\eqref{ODE}. 
\begin{theorem}[Existence] 
\label{tjklf-3-2}
Assume that
\begin{enumerate}[i)]
\item
  the functional 
  $\mathcal{H}_t \left(\mathcal{B}^{0,c}\right)$
  is integrable for all $t\in[0,T]$ and $c \in {\cal C}_f$,
  and
\item
  the sequence
$
  \big(  \widetilde{\mathcal{H}}_t^{f,n}
    \big(
    \mathcal{B}^{0,f}
    \big)
    \big)_{n\geq 0}
    $ defined in \eqref{functional-2} 
 is uniformly integrable for all $t\in [0,T]$.
\end{enumerate}
Then, the ODE~\eqref{ODE} has a solution $(x(t))_{t\in [0,T]} \in {\cal C}^1([0,T],\R^d)$
 which admits the probabilistic representation
  \begin{equation}
    \label{prob-rep-branching}
    x(t) = \E \left[ \mathcal{H}_t\left(\mathcal{B}^{0,\id}\right) \right], \quad t\in[0,T].
  \end{equation}
\end{theorem}
\begin{proof}
  By taking $c\in {\cal C}_f$ in Lemma~\ref{lemma-uc}, we see that the family
  $\{ x_c \}_{c\in \{\id\}\cup{\cal C}_f }$
  defined in \eqref{uc} solves the ODE system~\eqref{ode-system-a}-\eqref{ode-system-b}.
 Taking $g=f$ in Proposition~\ref{unf-intg}-\eqref{unf-intg-2},
 we have $f(x_\id(t)) = \E \left[ \mathcal{H}_t\left(\mathcal{B}^{0,f}\right) \right] = x_f(t)$ for all $t\in[0,T]$. It then follows from
 Lemma~\ref{lemma-uc} that $x_{\id}$ is a solution of the ODE~\eqref{ODE}.
\end{proof}
\noindent
 Under Assumption~\ref{asmpt-0},   
 Theorems~\ref{integ-1} and \ref{unf-integ-1} provide sufficient conditions on
 $f$ for the integrability required in Theorems~\ref{tjklf-3-2}
 and \ref{tjklf-3-1}. 
  By Theorems~\ref{tjklf-3-2} and \ref{integ-1}
   and the proof of Lemma~\ref{lemma-uc},
   under the conditions
   of Theorem~\ref{unf-integ-1} we also obtain the bound 
\begin{align*}
  x (t) %
  &= 
  x_0 + 
    \int_0^t
    \E \left[
      \mathcal{H}_t\left(\mathcal{B}^{s,f}\right)
      \right]
    ds
    \\
  & \leq 
  x_0 + 
    \int_0^t
  \frac{ e^{-\lambda s} C_0}{1 - (1-e^{-\lambda s}) C_0^{2}} 
    ds
    \\
  & = 
  x_0 + 
  \frac{t}{C_0} 
  + 
  \frac{1}{\lambda C_0}
  \log \frac{1}{C_0^{2} - ( C_0^{2} - 1 ) e^{\lambda t}},
  \qquad t\in [0,T]. 
\end{align*}
 where $C_0$ is defined in \eqref{eqrho}. 
\begin{theorem}[Uniqueness] 
\label{tjklf-3-1}
Assume that the ODE~\eqref{ODE} admits a solution $(x(t))_{t\in [0,T]} \in {\cal C}^1([0,T],\R^d)$, and
that 
\begin{enumerate}[i)]
\item the functional 
 $\mathcal{H}_t \left(\mathcal{B}^{0,c}\right)$
    is integrable for all $t\in[0,T]$ and $c \in {\cal C}_f$, and
  \item
    the sequence
 $\big( \mathcal{H}_t^{g,n}
        \big(
    \mathcal{B}^{0,g}
    \big)\big)_{n\geq 0}$
    defined in \eqref{functional-1}
 is uniformly integrable for all $t\in [0,T]$. 
\end{enumerate}
Then, $(x(t))_{t\in [0,T]} $ is the unique solution
      of the ODE~\eqref{ODE}, and it admits the
      probabilistic representation
      \eqref{prob-rep-branching}. 
\end{theorem}
\begin{proof}
  By taking $c\in {\cal C}_f$ in Lemma~\ref{lemma-uc}, we see that the family
  $\{ x_c \}_{c\in \{\id\}\cup{\cal C}_f }$
  defined in \eqref{uc} solves the ODE system~\eqref{ode-system-a}-\eqref{ode-system-b}.
  Taking $g=f$ in Proposition~\ref{unf-intg}-\eqref{unf-intg-i},
  we have the uniqueness for the ODE system~\eqref{ode-system-a}-\eqref{ode-system-b}. By Lemma~\ref{relation-ODE-ODEs},
  $\{ x_c \}_{c\in \{\id\}\cup{\cal C}_f }$
  coincides with $\{x, \nabla^m f(x) \ \! : \ \! m\geq 0 \}$.
  Thus, the solution $(x(t))_{t\in [0,T]}$
  admits the probabilistic representation $x = x_\id$,
  and it is unique. 
\end{proof}
\noindent
In Sections~\ref{sect4} and \ref{sect4-1}
we will provide sufficient 
conditions for the uniform integrability of
  $
  \big(  \widetilde{\mathcal{H}}_t^{f,n}
    \big(
    \mathcal{B}^{0,f}
    \big)
    \big)_{n\geq 0}
    $ and
    $\big( \mathcal{H}_t^{g,n}
        \big(
    \mathcal{B}^{0,g}
    \big)\big)_{n\geq 0}$
required in Proposition~\ref{unf-intg}
 and Theorems~\ref{tjklf-3-2}-\ref{tjklf-3-1}, 
 which also 
    imply the integrability of
    $\mathcal{H}_t\left(\mathcal{B}^{0,c}\right)$
    for $t\in[0,T]$ and $c \in {\cal C}_f$. 

\medskip 

 As a consequence,
 we have the following results 
 under uniform boundedness conditions on the derivatives of $f$.   
\begin{corollary}%
\label{t1}
 Let $\lambda, T>0$, and assume that for some $q>1$ we have 
\begin{equation}
\nonumber %
 |\nabla^k f (x_0)|
\leq
 \widebar{F}_\rho(T)
 \frac{2^{1/q}(1-e^{-\lambda T})^{1/(2q)}}{
 ( \sqrt{ 4 + e^{-2\lambda T} } - e^{-\lambda T} )^{1/q}}, 
 \quad \mbox{for all } k\geq 0.
\end{equation}
 Then, 
\begin{enumerate}[i)]
\item the ODE~\eqref{ODE} admits at most one solution
 in ${\cal C}^1([0,T],\R^d)$; 
\item
 if, in addition, 
 \begin{equation}
\label{cond-g-domain}
   \sup_{|x-x_0| < T / (1-e^{-\lambda T})^{1/(2q)}} \left| \nabla^m f(x) \right|
   < \frac{1}{(1-e^{-\lambda T})^{1/(2q)}},
   \quad \mbox{for all } m\geq 0, 
\end{equation} 
 then the ODE~\eqref{ODE} has a unique solution in ${\cal C}^1([0,T],\R^d)$, 
 which admits the probabilistic representation
 \eqref{prob-rep-branching}. 
\end{enumerate}
\end{corollary}
\begin{Proof}
 Part~$i)$ is a consequence of
 Theorem~\ref{tjklf-3-1}, 
 Theorem~\ref{integ-1}, and
 Theorem~\ref{unf-integ-1}-$i)$; 
 Part~$ii)$ is a consequence of
 Theorems~\ref{tjklf-3-2}-\ref{tjklf-3-1},  
 Theorem~\ref{integ-1}, and
 Theorem~\ref{unf-integ-1}-$ii)$.
\end{Proof} 
\noindent
 We note that as $q$ tends to $1$,
 Condition~\eqref{cond-g-domain} is compatible
 for $k=0$ with Theorem~2.3 in \cite{coddington},
 which shows that 
 if $f$ is $C$-Lipschitz on an interval
 $[x_0-b,x_0+b]$, then \eqref{ODE} admits a unique solution
 $x(t)$ on the time interval 
 $[0,b/\Vert f\Vert_\infty]$. 

 \medskip
  
 The next result relaxes the existence conditions of Corollary~\ref{t1}
by allowing the growth of derivatives of $f$,
as a consequence of Theorems~\ref{tjklf-3-2}-\ref{tjklf-3-1}
and \ref{integ-3-2}-\ref{unf-integ-3-2}. 
\begin{corollary} %
  Let $\lambda, T>0$, and assume that 
  for some $q>1$ and $\delta > 0$ we have 
  $$
    | f(x_0) |
  <
  \frac{e^{\lambda T}}{ ( 2 (1-e^{-\lambda T})\delta )^{1/(2q)}}
  \quad
  \mbox{and}
  \quad 
  |\nabla^k f(x_0) |
  \leq e^{\lambda T} (k \delta)^{1/(2q)}, 
   \quad k\geq 1.  
   $$ %
   Then,
   \begin{enumerate}[i)]
    \item 
   the ODE~\eqref{ODE}
   admits at most one solution
   in ${\cal C}^1([0,T],\R^d)$; 
 \item
  if, in addition, 
 \begin{equation}
   \nonumber %
  \sup_{|x-x_0| < T / (1-e^{-\lambda T})^{1/(2q)}} \left| \nabla^m f(x) \right| 
  <
  \frac{1}{(1-e^{-\lambda T})^{1/(2q)} },
  \quad \mbox{for all } m \geq 0, 
\end{equation}
 then the ODE~\eqref{ODE} has a unique solution in ${\cal C}^1([0,T],\R^d)$
 which admits the probabilistic representation
 \eqref{prob-rep-branching}. 
  \end{enumerate}
\end{corollary}
\begin{Proof}
 Part~$i)$ is a consequence of
 Theorem~\ref{tjklf-3-1}, 
 Theorem~\ref{integ-3-2}, and
 Theorem~\ref{unf-integ-3-2}-$i)$; 
 Part~$ii)$ is a consequence of
 Theorems~\ref{tjklf-3-2}-\ref{tjklf-3-1},  
 Theorem~\ref{integ-3-2}, and
 Theorem~\ref{unf-integ-3-2}-$ii)$ with $\delta = 1$ and $\gamma = 0$.
\end{Proof}

\section{Stochastic dominance of random binary trees} 
\label{sect5-1}
\noindent
 In this section, we present
 the stochastic dominance %
 results for branching processes, 
 that will be used in Section~\ref{sect4}. 
\noindent
 The sufficient conditions
 for the integrability of $\mathcal{H}_t \left(\mathcal{B}^{0,c}\right)$ 
 presented in Sections~\ref{sect4} and \ref{sect4-1}
 rely on an age-dependent continuous-time branching chain
 $( X_t )_{t\in[0,T]}$ generating a filtration
 $( \F_t )_{t\in [0,T]}$ on $(\Omega,\F,\P)$,
 and on its stochastic domination by 
 a continuous-time binary branching chain  
 $\big( \widetilde{X}_t\big)_{t\in [0,T]}$. 
\begin{definition}
  \begin{enumerate}[i)] 
  \item
    Let 
     \begin{equation}\label{discrete-chain}
  X_t := |\mathcal{K}_t^\pt| = 1 + \sum_{T_\mathbf k \leq t} \Delta X_{T_\mathbf k}
\end{equation}
 denote the binary branching chain starting from ${X}_0 = 1$ and 
 formed by the number of branches of $\mathcal{B}^{0,c}$ living
 at time $t>0$, 
 with offspring count 
 $1+\Delta X_{T_\mathbf k} = 2$ at any splitting time
 $T_\mathbf k$, 
 with lifetimes distribution ${\rho}$, and total progeny process 
\begin{equation}
  \label{sum-chain}
  N_t := |\mathcal{K}_t|
  = 1 + \sum_{T_\mathbf k \leq t} \left( 1+\Delta X_{T_\mathbf k} \right) = X_t + \sharp \{\mathbf k\in \mathbb K: T_\mathbf k \leq t\}, 
 \quad t\in [0,T]. 
\end{equation}
\noindent
\item
 Let $\big( \widetilde{X}_t\big)_{t\in [0,T]}$ be a 
 continuous-time binary branching chain  
 starting from $\widetilde{X}_0 = 1$,
 in which the lifetimes of branches are independent and
 identically distributed via a common
 exponential density function
$\widetilde{\rho} (t) = \lambda e^{-\lambda t}$, $t\geq 0$, 
 with parameter $\lambda > 0$,
 and progeny process $\widetilde{N}_t$. 
  \end{enumerate} 
\end{definition} 
   \noindent
 Recall that from e.g. Eq. (8) page 3 of \cite{kendall1948},
 \cite[Example 13.2 page 112]{harris1963}, 
 \cite[Example 5 page 109]{athreya},  
 the total progeny $\widetilde{N}_t$
 of $\big( \widetilde{X}_t\big)_{t\in \real_+}$ 
 has distribution 
\begin{equation}
\nonumber %
 {\P}\big( \widetilde{N}_t = m \big)
 =
\begin{cases}
 e^{-\lambda t} (1-e^{-\lambda t})^n, & m=2n+1, \ n \geq 0, %
\\
 0, & \text{otherwise}, \quad m \geq 0, 
\end{cases}
\end{equation}
 and probability generating function 
\begin{equation}
\label{pgf-ct}
G_t(z) = \E \big[ z^{\widetilde{N}_t} \big] %
= \frac{z e^{-\lambda t}}{1 - (1-e^{-\lambda t}) z^2}, \quad
 z < (1-e^{-\lambda t})^{-1/2}. 
\end{equation}
\noindent 
 Similarly to ${\mathcal{B}}^{0,c}$, 
 we consider the marked random branching tree
 $\widetilde{\mathcal{B}}^{0,j}$ constructed
 by assigning a mark $\tilde{c}_{\mathbf k}$ to each branch $\mathbf{k}$
 in the set $\widetilde{\mathcal{K}}_t$ of branches of
 the branching chain $\big( \widetilde{X}_t\big)_{t\in [0,T]}$, 
 in the following way:
 \begin{enumerate}[a)]
 \item
   the initial branch has label $\varnothing$ and is marked by
   $\tilde{c}_{\mathbf \varnothing} = j\in\N$;
 \item
   if a branch $\mathbf{k}$
   is marked by $\tilde{c}_{\mathbf k} = i\in\N$ and splits,
   its two children are
   respectively marked by $0$ and $i+1$. 
 \end{enumerate}

 \noindent
 Figure~\ref{fjkld23-3} 
 presents a sample of the corresponding
 random marked tree. %
 
\begin{figure}[H]
\centering
\tikzstyle{level 1}=[level distance=6cm, sibling distance=4cm]
\tikzstyle{level 2}=[level distance=6cm, sibling distance=5cm]
\tikzstyle{level 2}=[level distance=6cm, sibling distance=4cm]
\begin{center}
\resizebox{0.83\textwidth}{!}{
\begin{tikzpicture}[scale=1.0,grow=right, sloped]
        \node[rectangle,draw,black,text=black,thick, rounded corners=5pt,fill=gray!05] {\large $0$}
            child {
  node[name=data, %
         draw, rounded corners=5pt,fill=gray!05,xshift=1cm] (main) 
        {\large $T_{(1)}$ \nodepart{two} j}
                child{
                  node[%
                    draw,black,text=black,thick,yshift=0.4cm,rounded corners=5pt,fill=gray!05,xshift=2cm,yshift=-2cm]{\large $T_{(1,2)}$ \nodepart{two} j+1} 
                    child{
    node[name=data, %
         draw, rounded corners=5pt,yshift=0cm,fill=gray!05,xshift=4cm]  
        {\large $t$ \nodepart{two} j+2}
                    edge from parent
                    node[above,trapezium,draw,fill=gray!05]{$(1,2,2)$}
                    node[below,draw,rectangle,fill=gray!05]{\large $j+2$}%
                    }
                    child{
    node[name=data, %
         draw, rounded corners=5pt,yshift=-0.6cm,fill=gray!05,xshift=4cm]  
        {\large $t$ \nodepart{two} 0}
                    edge from parent
                    node[above,trapezium,draw,fill=gray!05]{$(1,2,1)$}
                    node[below,draw,rectangle,fill=gray!05]{\large $0$} %
                    }
                edge from parent
                node[above,trapezium,draw,fill=gray!05]{$(1,2)$}
                node[below,draw,rectangle,fill=gray!05]{\large $j+1$}%
                }
                child{
                  node[%
                    draw,black,text=black,thick,yshift=0.4cm,rounded corners=5pt,fill=gray!05]{\large $T_{(1,1)}$ \nodepart{two} 0} %
                    child{
                      node[%
                        draw,black,text=black,thick,yshift=0.4cm,rounded corners=5pt,fill=gray!05,xshift=1cm]{\large $T_{(1,1,2)}$ \nodepart{two} 1} 
                    child{
    node[name=data, %
         draw, rounded corners=5pt,yshift=0cm,fill=gray!05,xshift=-1cm]  
        {\large $t$ \nodepart{two} 2}
                    edge from parent
                    node[above,trapezium,draw,fill=gray!05]{$(1,1,2,2)$}
                    node[below,draw,rectangle,fill=gray!05]{\large $2$}%
                    }
                    child{
    node[name=data, %
         draw, rounded corners=5pt,yshift=-0.6cm,fill=gray!05,xshift=-1cm]  
        {\large $t$ \nodepart{two} 0}
                    edge from parent
                    node[above,trapezium,draw,fill=gray!05]{$(1,1,2,1)$}
                    node[below,draw,rectangle,fill=gray!05]{\large $0$} %
                    }
                    edge from parent
                    node[above,trapezium,draw,fill=gray!05]{$(1,1,2)$}
                    node[below,draw,rectangle,fill=gray!05]{\large $1$} %
                    }
                    child{
    node[name=data, %
         draw, rounded corners=5pt,fill=gray!05,xshift=6cm]  
        {\large $t$ \nodepart{two} 0}
                    edge from parent
                    node[above,trapezium,draw,fill=gray!05]{$(1,1,1)$}
                    node[below,draw,rectangle,fill=gray!05]{\large $0$} %
                    }
                edge from parent
                node[above,trapezium,draw,fill=gray!05]{$(1,1)$}
                node[below,draw,rectangle,fill=gray!05]{\large $0$} %
                }
                edge from parent
                node[above,trapezium,draw,fill=gray!05] {$(1)$}
                node[below,draw,rectangle,fill=gray!05]  {\large $j$} %
            };
\end{tikzpicture}
}
\end{center}
\vskip-0.3cm
\caption{Sample of
  $\widetilde{\mathcal{K}}_t$ for the random tree
  $\widetilde{\mathcal{B}}^{0,j}$ started from $j\in \inte$.}
\label{fjkld23-3}
\end{figure}
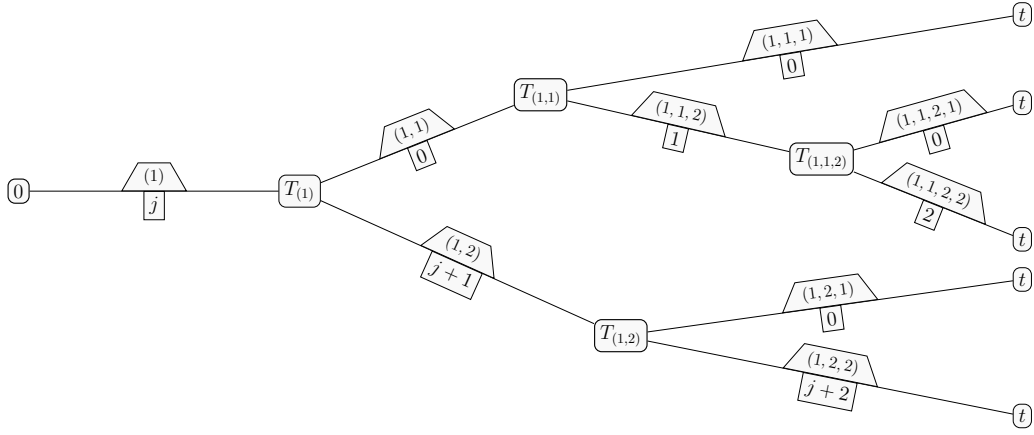

\vspace{-0.3cm}

\noindent
\begin{assumption}
\label{asmpt-0}
 There exists $\lambda > 0$ such that the lifetime probability density
 $\rho$ dominates the exponential distribution with parameter
 $\lambda$ in the sense that 
 \begin{equation}
   \label{asmpt-1-00}
   \widebar{F}_\rho(r) \geq
   e^{-\lambda r}
   = \widebar{F}_{\widetilde{\rho}} (r),
   \qquad r \geq 0. 
\end{equation}
\end{assumption}
\noindent
 An example of probability density function $\rho$
 satisfying Assumption~\ref{asmpt-0} is provided in 
 Lemma~\ref{rho-ext-1}. %
\begin{lemma}
\label{stoch-domin-1}
 Under Assumption~\ref{asmpt-0}, the processes
 $( X_t )_{t\in[0,T]}$
 and $( N_t )_{t\in[0,T]}$ are stochastically dominated by $\big( \widetilde{X}_t\big)_{t\in [0,T]}$ and
$\big( \widetilde{N}_t\big)_{t\in [0,T]}$ respectively,
i.e. 
  \begin{equation}
    \label{fjkd2}
\P \left( X_t \geq n \right) \leq \P \big( \widetilde{X}_t \geq n \big)
\quad
\mbox{and} 
\quad
 \P \left( N_t \geq n \right) \leq \P \big( \widetilde{N}_t \geq n \big),
    \quad n \geq 0,
    \
    t\in [0,T].
  \end{equation}
\end{lemma} 
\begin{proof}
 By \eqref{discrete-chain} we have 
\begin{equation*}
  \P \left( X_t \geq n \right) = \P \left( 1+ \sharp \{\mathbf k\in \mathbb K \ : \ T_\mathbf k \leq t\} \geq n \right)
  = \E [ \Phi((t_\mathbf k: \mathbf k\in \mathbb K)) ],
   \qquad n \geq 0, 
\end{equation*}
where the function
$$
\Phi: (t_\mathbf k \ : \ \mathbf k\in \mathbb K) \mapsto \ind_{\{1+ \sharp \{\mathbf k\in \mathbb K \ : \ T_\mathbf k \le t\} \ge n\}}
$$
 is non-increasing in every variable $t_\mathbf k$, $k\in \mathbb K$.
Assumption~\ref{asmpt-0} implies that for every $\mathbf k\in \mathbb K$ and $s\ge 0$,
\begin{equation*}
  \P\left( T_\mathbf k - T_{\mathbf{k}^-} \le s \right) \le 1-e^{-\lambda s} = \P\big(
   \widetilde{T}_\mathbf k - \widetilde{T}_{\mathbf{k}^-} \le s \big),
\end{equation*}
i.e., $t_\mathbf k = T_\mathbf k - T_{\mathbf{k}^-}$ stochastically dominates $\tilde{t}_\mathbf k = \widetilde{T}_\mathbf k - \widetilde{T}_{\mathbf{k}^-}$, 
hence from \cite[Theorem 6.B.3]{shaked2007stochastic},
 the random vector $(t_\mathbf k: \mathbf k\in \mathbb K)$ stochastically dominates $(\tilde{t}_\mathbf k: \mathbf k\in \mathbb K)$.
Thus, using the definition of stochastic dominance for random vectors \cite[Eq.~(6.B.4)]{shaked2007stochastic}, we get
\begin{align*}
  \P \left( X_t \ge n \right) & = \E \left[ \Phi((t_\mathbf k \ : \ \mathbf k\in \mathbb K)) \right] \\
  & \le \E \left[ \Phi((\tilde{t}_\mathbf k \ : \ \mathbf k\in \mathbb K)) \right]
  \\
  & = \P \big( 1+ \sharp \{\mathbf k\in \mathbb K \ : \ \widetilde{T}_\mathbf k \le t\} \ge n \big)
  \\
   & = \P \big( \widetilde X_t \ge n \big), 
\end{align*}
 and \eqref{fjkd2} follows similarly from \eqref{sum-chain}
 and its counterpart for $\widetilde{N}_t$.
\end{proof}
\noindent
Finally, from \cite{huangprivault1} we have the following
integrability result for the multiplicative progeny
of the random tree $\widetilde{\mathcal{B}}^{0,j}$ started from $j\geq 1$. 
\begin{prop}
    \label{radius-conv-weight-ct}
(Corollary~3.5 in \cite{huangprivault1}). 
    Let $t\geq 0$, 
    $j \geq 0$,
    $\delta >0$, $\gamma > 1 $,
    and let $(\sigma (k))_{k\geq 0}$ be a real sequence such that 
    \begin{equation}
\nonumber %
   0\leq \sigma (0) < \frac{1}{
      (1-e^{-\lambda t})\gamma \delta } 
  \quad
  \mbox{and}
  \quad 
   0\leq \sigma (k) < (k-2+\gamma)\delta , \quad k\geq 1. 
\end{equation} 
 Then, we have the bound 
$$ 
 \E_j \left[ 
   \prod_{\mathbf k \in \widetilde{\mathcal{K}}_t} \sigma
   (\tilde{c}_{\mathbf k}) 
 \right] 
 \leq
  \frac{e^{-\lambda t} \sigma (j)}
  {
    ( 1- (1-e^{-\lambda t}) \gamma \delta \sigma (0) )^{
      1 + (j-1)/\gamma 
    }
    } < \infty, 
$$
 where $\E_j$ denotes expectation given that
 the random tree $\widetilde{\mathcal{B}}^{0,j}$ is started from
 the mark $j\geq 1$. 
\end{prop} 
\noindent 
\section{Integrability - bounded marks} 
\label{sect4}
\noindent
The goal of this section and the next one is to
derive the integrability results
 needed for the application of Proposition~\ref{unf-intg} and
 Theorems~\ref{tjklf-3-2}-\ref{tjklf-3-1}. 
\noindent
We fix an initial mark $c = g\in {\cal C} \setminus \{ \id \}$, %
 in which case 
$\mathcal{B}^{0,g}$ is a binary branching chain
bearing marks of the form either $\nabla^m f$ or $\nabla^m g$, $m\geq 0$.
\noindent
In Theorems~\ref{integ-1} and \ref{unf-integ-1}
we improve on \cite[Proposition~4.3]{penent4},
 see also \cite[Theorem~3.3]{bakhtin}, via a detailed
 study of the 
 integrability of $\mathcal{H}_t \left(\mathcal{B}^{0,g}\right)$
 by investigating the probability generating function 
 of the cardinality $|\mathcal{K}_t|$.  
 In comparison with e.g. Theorem~4.2 of
 \cite{penent4}, Theorems~\ref{integ-1} and \ref{unf-integ-1}
 provide sufficient conditions on equation coefficients
 for the (uniform) integrability of the random
 weights involved in the probabilistic representation
 of ODE solutions. 
\begin{theorem}[Integrability I]
\label{integ-1}
 Let $q\ge1$ and $g \in {\cal C} \setminus \{ \id \}$. 
 Under Assumption~%
\ref{asmpt-0},
 suppose that %
\begin{equation}
    \label{eqrho} 
    1 < 
    C_0 : =
  \max \left(
  \sup_{k\geq 0} \frac{|\nabla^k f (x_0)|}{\widebar{F}_\rho(T)},
  \sup_{k\geq 0} \frac{|\nabla^k g (x_0)|}{\widebar{F}_\rho(T)} , \frac{1}{\rho_*(T)}
  \right)
  < \frac{1}{\big( 1 - e^{-\lambda T}\big)^{1/(2q)}}. 
\end{equation}
 where $\rho_*(T) := \inf_{s\in[0,T]} \rho (s) > 0$. Then, we have the %
 $L^q$ bound 
\begin{equation*}
  \E \left[ \left| \mathcal{H}_t \left(\mathcal{B}^{0,c}\right) \right|^q \right] \leq \frac{ e^{-\lambda t} C_0^q}{1 - (1-e^{-\lambda t}) C_0^{2q}},
  \quad
  t\in[0,T],
  \ 
  c\in{\cal C}_f\cup {\cal C}_g. 
\end{equation*}
\end{theorem}
\begin{proof}
 Under Assumption~\ref{asmpt-0}, 
 $(N_t)_{t\in [0,T]} := ( |\mathcal{K}_t| )_{t\in [0,T]}$
 is stochastically dominated by
 $\big( \widetilde{N}_t \big)_{t\in [0,T]}$
 by Lemma~\ref{stoch-domin-1}.
 Hence, since $x\mapsto C_0^x$ is increasing,
 we have
\begin{equation*}
  \E \left[ \left|\mathcal{H}_t\left(\mathcal{B}^{0,c}\right)\right|^q \right] \leq \E \big[ ( C_0^q)^{|\mathcal{K}_t|} \big] %
    \leq %
    \E \big[ ( C_0^q)^{\widetilde{N}_t} \big], 
\end{equation*}
  see Section~1.A.1 of \cite{shaked2007stochastic}.
  We conclude from \eqref{pgf-ct},
  which shows that 
 $\E \big[ ( C_0^q )^{\widetilde{N}_t} \big] < \infty$
 since $C_0<(1-e^{-\lambda T})^{-1/(2q)}$ under \eqref{eqrho}, %
 for $t\in[0,T]$. 
\end{proof}
\noindent
 Condition~\eqref{eqrho} involves constraints on both the nonlinearity $f$ and the lifetime density $\rho$. 
    \begin{remark}\label{lemma-rescaling}
                 Let $\mu >0$. For the rescaled ODE
        \begin{equation*}
    x_\mu' (t) = \frac{1}{\mu} f(x_\mu(t)), \quad t\in (0, \mu T],
  \end{equation*}
 with solution $x_\mu(t) := x(t/\mu)$, 
 $C_0$ in \eqref{eqrho} is replaced with 
   \begin{equation*}
 C_{0,\mu}: = \max \left(\sup_{k\geq 0} \frac{|\nabla^k f (x_0)|}{\mu \widebar{F}_\rho(\mu T)},
    \sup_{k\geq 0}
    \frac{|\nabla^k g (x_0)|}{\widebar{F}_\rho(\mu T)} ,
    \frac{1}{\rho_*(\mu T)}
    \right). 
  \end{equation*}
  In this case, the constraint on $\rho$ is now realized
  as an upper bound on $\mu T$,
  hence a smaller $\mu$ close to zero yields
  a looser constraint on $T$, but can 
  translate into a stricter constraint on~$f$.
\end{remark}
    \noindent
     In Lemma~\ref{rho-ext-1} we show 
 the existence of a probability density function $\rho: [0,\infty) \to \R$
    satisfying both Assumption~\ref{asmpt-0}
    and Condition~\eqref{eqrho}, resp. \eqref{eqrho2}, 
    using an upper bound on the existence time $T$ of the solution.
\begin{theorem}[Uniform integrability I]
\label{unf-integ-1}
 Let $q\ge1$ and $g \in {\cal C} \setminus \{ \id \}$. 
 Under Assumption~%
\ref{asmpt-0} and \eqref{eqrho},
 suppose in addition that
  \begin{equation}
    \label{eqrho2} 
     1 < C_0^q < 
 \frac{
   \sqrt{ 4 + e^{-2\lambda T} } - e^{-\lambda T} }{2\sqrt{1-e^{-\lambda T}}}, 
\end{equation} 
 where $C_0$ is defined in \eqref{eqrho}. 
 Then, $\mathcal{H}_t^{g,n}$ 
 defined in \eqref{functional-1} is $L^q$-integrable,
 uniformly in $t\in[0,T]$ and $n\geq 0$.
 Moreover, 
\begin{enumerate}[i)] 
\item
 if $q>1$, then $\big( \mathcal{H}_t^{g,n} \big)_{(n,t)\in \inte \times \real_+}$ is
 uniformly integrable; %
\item
 if %
\begin{equation}\label{cond-g-domain-2}
  \sup_{|x-x_0| < T (1-e^{-\lambda T})^{-1/(2q)}} \left| \nabla^m g(x) \right|
  < \frac{1}{(1-e^{-\lambda T})^{1/(2q)}},
  \quad
  \mbox{for all } m\geq 0, 
\end{equation}
then $\widetilde{\mathcal{H}}_t^{g,n}$ defined in \eqref{functional-2} is $L^q$-integrable, uniformly in $t\in[0,T]$ and $n\geq 0$. 
\end{enumerate}
\end{theorem}
\begin{proof}
\noindent
$i)$ One can check that Condition~\eqref{eqrho}
 is satisfied under \eqref{eqrho2}, 
hence by Theorem~\ref{integ-1},
$\mathcal{H}_t\left(\mathcal{B}^{0,c}\right)$ is $L^q$-integrable,
uniformly in $t\in[0,T]$ and $c\in{\cal C}_f\cup {\cal C}_g$. It also follows
from Lemma~\ref{lemma-uc} that the family
$$
 \{ x_c \}_{c\in {\cal C}_f\cup{\cal C}_g} = 
 \{ \E [ \mathcal{H}_t(\mathcal{B}^{0,c}) ] \}_{c\in {\cal C}_f\cup{\cal C}_g}
 $$
solves the ODE system~\eqref{ode-system-g}. On the other hand, by
Theorem~\ref{integ-1} %
we have 
  \begin{align}
    \label{eqn-11}
    |x_c (t)|^q & \leq  \E \left[ \left| \mathcal{H}_t\left( \mathcal{B}^{0,c} \right) \right|^q \right]
    \\
    \nonumber
    & \leq 
    \ind_{\{C_0\leq 1\}} + \frac{C_0^q e^{-\lambda T}}{1 - (1-e^{-\lambda T}) C_0^{2q}} \ind_{\{C_0> 1\}}. 
\end{align}
  Under \eqref{eqrho2} this yields
$$
  |x_c (t)|^q <  \frac{1}{\sqrt{1-e^{-\lambda T}}}, 
    \qquad
    c\in {\cal C}_f\cup{\cal C}_g, \ t\in[0,T], 
$$
 hence it follows from \eqref{functional-1} that
  \begin{equation*}
    \left| \mathcal{H}_t^{g,n} 
            \big(
    \mathcal{B}^{0,g}
    \big)
    \right|^q \leq ( C_1^q)^{ | \mathcal{K}_t^{n+1} |} \leq
    ( C_1^q)^{\left| \mathcal{K}_t \right|},  
  \end{equation*}
 where \begin{equation*}
    C_1^q := \max \left(
    C_0^q ,
    \sup_{c\in {\cal C}_f\cup{\cal C}_g, \ \! t\in [0,T]} |x_c (t)|^q
    \right)
    < \frac{1}{\sqrt{1-e^{-\lambda T}}}. 
\end{equation*}
 From \eqref{pgf-ct}, we conclude that 
 $\mathcal{H}_t^{g,n}$ is $L^q$-integrable, 
  uniformly in $t\in[0,T]$ and $n\geq 0$, 
  as in the proof of Theorem~\ref{integ-1}.
  
  \noindent 
  $ii)$ By point~$i)$ above, the family
  $$
  \{ x_c \}_{c\in {\cal C}_f\cup\{\id\}}
  = 
  \{ \E [ \mathcal{H}_t(\mathcal{B}^{0,c}) ] \}_{c\in {\cal C}_f\cup\{\id\}}
  $$
  solves the ODE system~\eqref{ode-system-a}-\eqref{ode-system-b},
  and by \eqref{eqn-11} and the Cauchy--Schwartz inequality we have 
  \begin{equation}
    \label{eqn-20}
    |x_{\id}(t) - x_0| \leq t^{1-1/q} \left( \int_0^t \left| x_f(s) \right|^q ds \right)^{1/q}
        < \frac{T}{(1-e^{-\lambda T})^{1/(2q)}},
    \qquad t\in[0,T]. 
\end{equation} 
 In addition, by \eqref{functional-2} we have 
\begin{equation*}
  \big| \widetilde{\mathcal{H}}_t^{g,n}
          \big(
    \mathcal{B}^{0,g}
    \big)
    \big|^q \leq ( C_2^q)^{ | \mathcal{K}_t^{n+1} |} \leq (C_2^q)^{\left| \mathcal{K}_t \right|}, 
  \end{equation*}
  where 
  \begin{equation*}
    C_2^q :=
    \max \left(
    C_0^q ,
    \sup_{c\in {\cal C}_f, \ \! t\in [0,T]} |x_c (t)|^q
    ,
    \sup_{c\in{\cal C}_g, \ \! t\in[0,T]} \left| c \left( x_\id(t) \right) \right|^q
    \right)
    <
    \frac{1}{\sqrt{1-e^{-\lambda T}}}, 
  \end{equation*}
  due to 
  \eqref{cond-g-domain-2}
  and
  \eqref{eqn-20}.
  As above, we conclude from \eqref{pgf-ct}. 
\end{proof}
\begin{remark}
  When $g=f$,
  Condition~\eqref{cond-g-domain-2} provides a quantitative
  estimate for the neighborhood $U$ in
  Condition~\hyperlink{asmp-lp}{(C2)}. %
\end{remark}
\section{Integrability - unbounded marks} 
\label{sect4-1}
\noindent
 In this section we consider a weight function
 $\sigma (k)$
 that depends only on the order $k$
 of the gradients %
 $\nabla^k f$, $\nabla^k g$, $k\geq 0$. 
\begin{theorem}[Integrability II] %
   \label{integ-3-2}
 Let $T>0$, $q\ge1$, $\delta>0$, and $\gamma\geq 2$.
 Under Assumption~\ref{asmpt-0}, suppose that
\begin{equation}
\label{eqrho-1} 
 \rho_*(T)
 > \big( 1 - e^{-\lambda T}\big)^{1/(2q)}, 
\end{equation}
 and that 
 \begin{equation}
   \label{C0-sigma-2}
   \sigma(k):= \frac{|\nabla^k f(x_0) |}{\widebar{F}_\rho(T)}\vee \frac{|\nabla^k g(x_0) |}{\widebar{F}_\rho(T)}, \qquad k\geq 0,
\end{equation}
 satisfies 
 \begin{equation}
   \label{upper-unbound-1-2}
      1 < 
      {\sigma(0)}^{2q} < \frac{1}{(1-e^{-\lambda T}) \gamma \delta } 
  \quad
  \mbox{and}
  \quad 
      1 < 
 {\sigma(k)}^{2q} \leq (k-2+\gamma) \delta , \quad k\geq 1. 
\end{equation}
 Then, 
 we have the $L^q$ bound 
\begin{equation*}
  \E \left[ \left| \mathcal{H}_t\left(\mathcal{B}^{0,c}\right) \right|^q \right] \le
  e^{-\lambda t} \sigma(j)^q 
  \frac{
\big( 1-  (1-e^{-\lambda t})
    \sigma(0)^{2q} \gamma \delta \big)^{ - 1/2 - (j-1)/(2\gamma ) }
  }{\big( 1 - (1-e^{-\lambda t}) / \rho_*(T)^{2q} \big)^{1/2}},
  \quad
  c \in \{ \nabla^j f , \nabla^j g \},
\end{equation*}
 for each $j\geq 0$ and $t\in [0,T]$. 
\end{theorem}
\begin{proof}
    Since $c \ne \id$ we have $|\mathcal K_t^\circ| = (|\mathcal K_t|-1)/2$,
    hence from \eqref{functional} we have
\begin{equation*}
\left|\mathcal{H}_t \left(\mathcal B^{0,c}\right)\right| \le \rho_*(T)^{-|\mathcal K_t^\circ|} \prod_{\mathbf k \in \mathcal{K}_t^\pt} \sigma(c^\mathbf k) \le \rho_*(T)^{-(|\mathcal K_t|-1)/2} \prod_{\mathbf k \in \mathcal{K}_t} \sigma(c^\mathbf k).
\end{equation*}
Thus, 
 \begin{align*}
  \E \left[ \left|\mathcal{H}_t\left(\mathcal{B}^{0,c}\right)\right|^q \right]
  & \leq \E \left[
    \rho_*(T)^{
    -(|\mathcal{K}_t|-1)/2
      }
      \prod_{\mathbf k \in \mathcal{K}_t} \sigma(c_{\mathbf{k}} )
  \right]
  \\ 
  & \leq \E \left[
    \rho_*(T)^{ -( |\mathcal{K}_t|-1 ) q/ 2} \left(
    \prod_{\mathbf k \in \mathcal{K}_t} \sigma(c_{\mathbf{k}})
    \right)^q \right]
  \\
  & \leq
  \big(
  \rho_*(T)^q \E \big[ \rho_*(T)^{-q|\mathcal{K}_t|} \big]
  \big)^{1/2}
  \left( \E \left[
    \prod_{\mathbf k \in \mathcal{K}_t} \sigma(c_{\mathbf{k}})^{2q}
     \right] \right)^{1/2}. 
\end{align*}
 For the same reason as in Lemma~\ref{stoch-domin-1},
  we deduce that $\prod_{\mathbf k \in \mathcal{K}_t} \sigma(c^\mathbf k)$ is stochastically dominated by $\prod_{\mathbf k \in \widetilde{\mathcal{K}_t}} \sigma(\tilde{c}_\mathbf k)$.
  It then follows that
  $$\E [|H_t(B^{0,c})|^q] \le
  \big(\rho_*(T)^q E\big[\rho_*(T)^{-q \widetilde{N}_t}\big]\big)^{1/2}
  \left(
  \E\left[\prod_{\mathbf k \in \widetilde{\mathcal{K}_t}} \sigma(\tilde{c}_\mathbf k)^{2q}\right]
  \right)^{1/2}, \quad t\in [0,T], 
$$
  and we conclude from \eqref{pgf-ct} and
  Proposition~\ref{radius-conv-weight-ct}. 
\end{proof}
\noindent
 An example of a nonlinearity satisfying the growth condition
  \eqref{upper-unbound-1-2}
  is given by $f(x) = x \cos x$. 
\begin{theorem}[Uniform integrability II] %
\label{unf-integ-3-2}
 Let $T>0$, $q\ge1$, $\delta>0$, $\gamma\geq 2$,
 and $g \in {\cal C} \setminus \{ \id \}$. 
 Under Assumption~\ref{asmpt-0}, %
 suppose that \eqref{eqrho-1} holds
 and that the weight function
 $\sigma$ in \eqref{C0-sigma-2}
 satisfies %
 \begin{equation}
   \nonumber %
      1 < {\sigma(0)}^q < \frac{1}{(1-e^{-\lambda T})\gamma \delta }
\end{equation}
 and 
\begin{equation}\label{upper-unbound-unf-dc-2}
  1 < \sigma(k)^q < %
 e^{\lambda T} \sqrt{\frac{
  1 - \left(1-e^{-\lambda T}\right) / \rho_*(T)^{2q} 
  }{
  1-e^{-\lambda T}
}}
  \big( 1- (1-e^{-\lambda T}) 
  \sigma(0)^q \gamma \delta \big)^{1/2 + (k-1)/(2\gamma )},
  \quad
  k\geq 1. 
\end{equation}
Then, 
$\mathcal{H}_t^{g,n}$ defined in \eqref{functional-1} is $L^q$-integrable, uniformly in $t\in[0,T]$ and $n\geq 0$.
Moreover,
\begin{enumerate}[i)] 
\item 
 if $q>1$, then $\big( \mathcal{H}_t^{g,n} \big)_{(n,t)\in \inte \times \real_+}$ is uniformly integrable; 
\item 
 if 
\begin{equation}
  \label{cond-g-domain-3}
  \sup_{|x-x_0| < T / (1-e^{-\lambda T})^{1/(2q)}} \left| \nabla^m g(x) \right|
  < \frac{1}{(1-e^{-\lambda T})^{1/(2q)}},
  \quad
  \mbox{for all } m\geq 0, 
\end{equation} 
 then $\widetilde{\mathcal{H}}_t^{g,n}$ defined in \eqref{functional-2} is $L^q$-integrable, uniformly in
 $t\in[0,T]$ and $n\geq 0$. 
\end{enumerate}
\end{theorem}
\begin{proof}
$i)$ We note that Condition~\eqref{upper-unbound-unf-dc-2} implies \eqref{upper-unbound-1-2}, hence $(i)$ follows from Theorem~\ref{integ-3-2}. 

  \noindent
  $ii)$ Next, Lemma~\ref{lemma-uc} %
 and
 Theorem~\ref{integ-3-2} imply that
the family $(x_c = \E [ \mathcal{H}_t(\mathcal{B}^{0,c}) ]: c\in {\cal C}_f\cup\{\id\})$ solves the ODE system~\eqref{ode-system-a}-\eqref{ode-system-b}, 
 and 
  the family $(x_c = \E [ \mathcal{H}_t(\mathcal{B}^{0,c}) ]: c\in {\cal C}_f\cup{\cal C}_g)$ solves the ODE system~\eqref{ode-system-g}. 
  Moreover, by Theorem~\ref{integ-3-2}, for $c \in \{\nabla^k f ,
  \nabla^k g \}$, $k\geq 0$, and all $t\in[0,T]$,
  we have 
$$ 
    |x_c(t)|^q \leq \E \left[ \left| \mathcal{H}_t\left( \mathcal{B}^{0,c} \right) \right|^q \right]
  \leq %
    e^{-\lambda T} \sigma(k)^q
    \frac{
    \big( 1-
    (1-e^{-\lambda T}) 
    \sigma(0)^q \gamma \delta
    \big)^{
      - 1/2
      - (k-1)/(2\gamma )    
}}{{\big( 1 - \left(1-e^{-\lambda T}\right)
      / \rho_*(T)^{2q} \big)^{1/2}
    }}, %
$$ 
 hence it follows from \eqref{upper-unbound-unf-dc-2} that
  \begin{equation*}
    |x_c(t)|^q < \left(1-e^{-\lambda T}\right)^{-1 / 2}, \quad
    c \in \{ \nabla^k f, \nabla^k g \}, \quad k\geq 0, 
  \end{equation*}
  hence \eqref{eqn-20} and 
  \eqref{cond-g-domain-3} hold.
 Hence, as in the proof of Theorem~\ref{integ-3-2} 
 we have 
\begin{equation*}
   \big| \widetilde{\mathcal{H}}_t^{g,n} \big|^q \leq ( C_2^q)^{ | \mathcal{K}_t^{n+1} |} \leq (C_2^q)^{\left| \mathcal{K}_t \right|}, 
  \end{equation*}
 with 
  \begin{align*}
    C_2^q & :=
    \max \left(
    \sup_{k\geq 1} \sigma(k)^q
    ,
    \sup_{c \in \{ \nabla^k f, \nabla^k g \}, \ \! t\in [0,T]} |x_c(t)|^q
    ,
    \sup_{c = \nabla^k g, \ \! t\in[0,T]} \left| c \left( x_\id(t) \right) \right|^q
    \right)
    \\
     & <
    (1-e^{-\lambda T})^{-1/2}, 
  \end{align*}
 and we conclude from \eqref{pgf-ct}. 
\end{proof}

\begin{remark}
It can be verified that an explicit and sufficient condition for \eqref{upper-unbound-unf-dc-2} to hold is
\begin{align*}
 &    1 < 
 \sigma(0)^q \leq \frac{\left(1-e^{-\lambda T}\right)^{-1/2} e^{\lambda T}
      (1 - \left(1-e^{-\lambda T}\right) / \rho_*(T)^{2q} )^{1/2}}{1 + (1-e^{-\lambda T})^{1/2} \gamma \delta
      e^{\lambda T} \big( 1 - (1-e^{-\lambda T}) / \rho_*(T)^{2q} \big)^{1/2}},
\\
 &   1 < 
  \sigma(k)^q \leq \frac{\left(1-e^{-\lambda T}\right)^{-1/2} e^{\lambda T}
  ( 1 - \left(1-e^{-\lambda T}\right) / \rho_*(T)^{2q} )^{1/2}}{\big(1 +
  (1-e^{-\lambda T})^{1/2} e^{\lambda T} \gamma \delta
  ( 1 - \left(1-e^{-\lambda T}\right) / \rho_*(T)^{2q} )^{1/2}\big)^{
    - 1/2
    - (k-1)/(2\gamma )
}}, \quad k\geq 1.
\end{align*}
\end{remark}

\section{Examples} %
\label{sect_ex}
\noindent
In this section we construct
an example of probability density function 
satisfying Assumption~\ref{asmpt-0}. 
 In what follows, we let 
 \begin{equation*}
   C_1(q;T) := \frac{1}{(1- e^{- 2 \lambda T} )^{1 /(2q)}} 
   \mbox{ ~and~ } 
  C_2(q;T) := 
   \frac{( \sqrt{4 + e^{-2 \lambda T}} - e^{-\lambda T} )^{1/q}}{
     2^{1 / q} (1- e^{-\lambda T} )^{1 / (2q)}}. 
\end{equation*}
 The inequality \eqref{asmpt-1-00}, %
 together with
 Conditions~\eqref{eqrho} and \eqref{eqrho2}, imply 
\begin{equation*}
  1 = \widebar{F}_\rho(0) = \int_0^T \rho(r) dr + \widebar{F}_\rho(T)
  > \frac{T}{C_i (q;T)} + e^{-\lambda T} , 
\end{equation*}
hence
$$
 T < ( 1 - e^{-\lambda T} ) C_i(q;T),
\quad i=1,2.
$$ 
 Given that
 for any $q\ge 1$
 the functions $(0,\infty ) \ni T \mapsto (1 - e^{-\lambda T}) C_i(q;T)$,
 $i=1,2$, take values in the whole interval $(0,1)$,
 we have the following. 
 \begin{lemma}
   \label{rho-ext-1}
  Let $T>0$, $q\ge 1$, and $i=1$, resp. $i=2$.
  If $T<1$, then there exists a probability density function $\rho: [0,\infty) \to \R$ satisfying
    Assumption~\ref{asmpt-0} and 
 $1 / \rho_*(T) < C_i(q;T)$ for $i=1$, resp. $i=2$. 
\end{lemma}
\begin{proof}
\noindent 
 $i)$ Case 1: $C_i(q;T)
 > e^{\lambda T} / \lambda$.
 In this case, we let
 $$
 \rho(t) := \lambda e^{-\lambda t},
 \qquad
 t\geq 0,
 $$
 for some $\lambda > 0$. Then Assumption \ref{asmpt-0}
 trivially holds, and $1 / \rho_*(T) = e^{\lambda T} / {\lambda} <
 C_i(q;T)$. 

 \noindent 
 $ii)$ Case 2:
 $C_i(q;T) \le e^{\lambda T}/\lambda$. In this case, we let 
  \begin{equation*}
    \rho(t) :=
    \begin{cases}
      \displaystyle \frac{1}{C_i(q;T) - \e}, & 0\le t\le T,
      \medskip
      \\
      \lambda_2 e^{-\lambda_1 t}, & t>T, 
    \end{cases}
  \end{equation*}
  for some $\lambda, \lambda_1, \lambda_2 > 0$ and $\e>0$. Then the condition $1 / \rho_*(T) < C_i(q;T)$ trivially holds.
   Since $\rho$ is a probability density, we have 
  \begin{equation}\label{eqn-16}
    1 = \int_0^T \rho(t) dt + \int_T^\infty \rho(t) dt = \frac{T}{C_i(q;T) - \e} + \frac{\lambda_2}{\lambda_1} e^{-\lambda_1 T}, 
  \end{equation}
  hence
  $$
 \frac{\lambda_2}{\lambda_1} e^{-\lambda_1 T} = 1 - \frac{T}{C_i(q;T) - \e}.
  $$ 
  Since
  \begin{equation*}
    \widebar{F}_\rho(t) = 
    \begin{cases}
      \displaystyle
      \frac{T-t}{C_i(q;T) - \e}
      + \frac{\lambda_2}{\lambda_1} e^{-\lambda_1 T}, & 0\le t\le T,
      \medskip
      \\
      \displaystyle
      \frac{\lambda_2}{\lambda_1} e^{-\lambda_1 t}, & t>T, 
    \end{cases}
  \end{equation*}
  Assumption \ref{asmpt-0} amounts to 
\begin{equation} 
  \lambda_1 \le \lambda, \quad \frac{\lambda_2}{\lambda_1} e^{-\lambda_1 T} \ge e^{-\lambda T}, \label{eqn-17}
\end{equation}
and
\begin{equation} 
    \frac{T-t}{C_i(q;T) - \e}  + \frac{\lambda_2}{\lambda_1} e^{-\lambda_1 T} \ge e^{-\lambda t}, \qquad 0\le t\le T. \label{eqn-18}
  \end{equation}
  By applying \eqref{eqn-16} and the second inequality of \eqref{eqn-17}, we get
  \begin{equation}\label{eqn-19}
    T \le \left( 1 - e^{-\lambda T} \right) \left( C_i(q;T) - \e \right).
  \end{equation}
  By applying \eqref{eqn-16} and \eqref{eqn-18}, we get
  \begin{equation*}
    \frac{t}{1-e^{-\lambda t}} \le C_i(q;T) - \e, \quad \text{for all } 0\le t\le T,
  \end{equation*}
  which amounts again to \eqref{eqn-19}. In summary, the restriction that $\rho$ is a probability density and Assumption \ref{asmpt-0} together are equivalent to the three conditions: \eqref{eqn-16}, the first inequality of \eqref{eqn-17}, and \eqref{eqn-19}. Since $(1 - e^{-\lambda T} ) C_i(q;T)$, $i=1,2$ take values in the whole interval $(0,1)$, and recalling the assumption $T<1$, we deduce that there exist $\e>0$ and $\lambda>0$ such that \eqref{eqn-19} holds. Finally, one can choose appropriate $\lambda_1, \lambda_2 > 0$ satisfying the identity \eqref{eqn-16}. These define a probability density $\rho$ satisfying Assumption \ref{asmpt-0} and $1 / \rho_*(T) < C_i(q;T)$, as required.
\end{proof}
\noindent 
\begin{remark}
  \label{rmk-exp}
  \begin{enumerate}[i)]
\item 
  Note that in the first case of the above
  proof, i.e. $C_i(q;T) >
  e^{\lambda T} / \lambda$, we have
\begin{equation}
  \label{fhklf111}
  T <
    \lambda T e^{-\lambda T} C_i(q;T), 
\end{equation}
 and the suprema of the above right hand sides
 belong to $(1/e,1/2)$ for $i=1,2$. 
 Hence, when $T\in (0,1/e)$ 
 there exists $\lambda > 0$ such that
 \eqref{fhklf111} holds,
 and one can thereby choose $\rho$ as 
 the exponential density 
 $$
 \rho(t) := \lambda e^{-\lambda t},
 \qquad
 t\geq 0. 
$$
 In this case, conditions \eqref{eqrho}
 and \eqref{eqrho2} reduce respectively to
 $$
 \max
 \left(\sup_{m\geq 0} |\nabla^m f (x_0)| ,
 \sup_{m\geq 0} |\nabla^m g (x_0)| \right)
 < e^{-\lambda T} C_i(q; T),
 \quad i=1,2,
 $$
 where the above right hand sides are both increasing in
 $e^{-\lambda T} \in (0,1)$. 
 Thus, if $T$ is fixed in $(0,1/e)$
 then a smaller $\lambda>0$ yields looser 
 constraints on the nonlinearity,
 while if $\lambda >0$ is fixed,
 then a larger $T\in (0,1/e)$
 requires stricter constraints on the nonlinearity.
\item
 As seen in Remark~\ref{lemma-rescaling}, the
      upper bounds on $T$ can be adjusted by a rescaling.
      In this case,
      the condition on the nonlinearity $f$ will be
      adjusted accordingly. 
\end{enumerate}
\end{remark}
\begin{corollary}[Existence intervals]
Let $q>1$. Let $\lambda_0>0$ be the smallest solution of 
\begin{equation}\label{eqn-lambda}
  \lambda_0 e^{-\lambda_0/e} C_2 (e^{-\lambda_0/e}; q) = 1.
\end{equation}
 Under the conditions 
\begin{equation}\label{eqn-26}
 T < \frac{1}{e \lambda_0 |\nabla^m f (x_0)|}, \quad \mbox{for all } m \geq 0, 
\end{equation}
 and 
\begin{equation}\label{eqn-27}
  \sup_{|x-x_0| < T / (1-e^{-\lambda_0 T})^{{1}/{(2q)}}} \left| \nabla^m f(x) \right| < \left(1-e^{-\lambda_0 T}\right)^{-{1}/{(2q)}}, \quad
   \mbox{for all } m\geq 0, 
\end{equation}
then the ODE \eqref{ODE} has a classical solution $x\in C^1([0,T],\R^d)$ which
 admits the probabilistic representation \eqref{prob-rep-branching}.
\end{corollary}
\begin{proof}
 Let 
\begin{equation}\label{eqn-24}
  \mu_0 := \lambda_0 \sup_{m\geq 0} |\nabla^m f (x_0)|
  < \frac{1}{eT}.
\end{equation} 
 By definition of $\lambda_0$, for $\lambda$ sufficiently close to
 but smaller than $\lambda_0$, it holds that
\begin{equation*}
  \mu_0 T < \frac{1}{e} = \frac{\lambda_0}{e} e^{-\lambda_0/e} 
  C_2 (e^{-\lambda_0/e}; q) < \lambda \mu_0 T e^{-\lambda \mu_0 T} C_2 (e^{-\lambda \mu_0 T}; q),
\end{equation*}
that is 
\begin{equation}\label{eqn-25}
  C_2(e^{-\lambda \mu_0 T}; q) > \frac{1}{\lambda} e^{\lambda \mu_0 T}, 
\end{equation}
 hence from Remark~\ref{rmk-exp}-$(i)$ we can choose $\rho$ to be
 the exponential density $\rho(t) = e^{-\lambda t}, t\ge 0$.
 It then follows from Remark \ref{lemma-rescaling} with $\mu = \mu_0$ that 
 Condition~\eqref{eqrho2} with $g=f$ reduces to
  \begin{equation*}
   \frac{1}{\mu_0} \sup_{m\geq 0} |\nabla^m f (x_0)| < e^{-\lambda \mu_0 T} C_2 (e^{-\lambda \mu_0 T}; q),
  \end{equation*}
  which is fulfilled due to our
  choice of $\mu_0$ in \eqref{eqn-24} and \eqref{eqn-25}.
  We conclude from Theorem~\ref{unf-integ-1}-$(ii)$.
\end{proof}
\begin{remark}\label{rmk-lambda}
  It is easy to deduce that the solution $\lambda_0$ of \eqref{eqn-lambda} satisfies ${3}/{2} < \lambda_0 < e$, and it is increasing as $q\in (1, \infty)$ increases. When $q$ goes to infinity, $\lambda_0$ tends to $e$.
\end{remark}
\begin{example}
  Consider the monomial nonlinearity $f(x) = x^n$ for
  some $n\geq 2$, with initial data $x_0 > 0$.
  The solution of \eqref{ODE} is given by
  $$
  x(t) =
  \frac{1}{\big(x_0^{-1/(n-1)}- (n-1)t\big)^{1/(n-1)}},
  \qquad 0\le t< \frac{1}{(n-1)x_0^{n-1}}. 
  $$
  Letting $n_0: = \min ( n , \lfloor x_0\rfloor )$,
  we have $n_0\leq x_0< n_0+1\leq n$
  or $n=n_0 \leq x_0$, hence 
\begin{align*} 
 a(x_0;n_0) & := \max_{0\leq m \leq n} |\nabla^m f (x_0)|
\\
&
= \max_{0\leq m \leq n} \frac{n!}{m!} x_0^m
\\
&
= \frac{n!}{n_0!} x_0^{n_0}
 \\
  & \geq 2.
\end{align*} 
 Condition \eqref{eqn-26} becomes 
\begin{equation}\label{eqn-30-0}
  T < \frac{1}{ e \lambda_0 a(x_0;n_0)}, 
\end{equation}
which implies $\lambda_0 T<1/(2e)$.
 It also follows from Remark \ref{rmk-lambda} that
$$
\frac{T}{(1-e^{-\lambda_0 T})^{1/(2q)}} <
 \frac{1}{2e\lambda_0 (1-e^{-1/(2e)})^{1/(2q)}} <
   1,
   $$
 and we have 
$$
  \sup_{|x-x_0| < T / (1-e^{-\lambda_0 T})^{{1}/{(2q)}}
    \atop 
    0 \leq m \leq n} \left| \nabla^m f(x) \right|
    <
  \frac{n!}{n_0!}
  \big(
  x_0
  + T(1-e^{-\lambda_0 T})^{-1/(2q)} \big)^{n_0+1}
   < \frac{n!}{n_0!}
  (x_0+1)^{n_0+1}. 
$$ 
 Therefore, Conditions~\eqref{eqn-26} and \eqref{eqn-27} hold under
 \eqref{eqn-30-0} and %
 $$
 T < 
 - \frac{1}{\lambda_0}
 \log \left( 1 - \left( \frac{n_0!}{n!}\right)^{2q}
 (x_0+1)^{-2(n_0+1)q}
 \right)
 < \frac{1}{(n-1)x_0^{n-1}}.  
 $$
 In particular, when $n=2$ and $x_0=1$ we have $n_0=1$ and $a(x_0;n_0)=2$,
 hence 
 $$
 T <
 \frac{2}{3} 
 \min \left(
 \frac{1}{2e},  
 - 
 \log \left( 1 - \frac{1}{2^{6q}} 
 \right)\right).  
$$ 
\end{example}
\begin{example}
 Consider the exponential nonlinearity $f(x) = e^x$,
 and initial data $x_0 >0$. The solution of \eqref{ODE} is given by
$$
 x(t) = -\log ( e^{-x_0}- t), \quad 0\le t< e^{-x_0}, 
$$
 We have
$$ 
 a(x_0;n_0) := \sup_{ m \geq 0} |\nabla^m f (x_0)| = e^{x_0}.
$$ 
 Condition \eqref{eqn-26} becomes 
\begin{equation}\label{eqn-30-0-1}
  T < \frac{1}{\lambda_0e^{x_0+1}}, 
\end{equation}
which implies $\lambda_0 T<1/e$,
and it follows from Remark \ref{rmk-lambda} that
$$
\frac{T}{(1-e^{-\lambda_0 T})^{1/(2q)}} <
 \frac{1}{e\lambda_0 (1-e^{-1/e})^{1/(2q)}} <
   1,
   $$
 and we have 
$$ 
  \sup_{|x-x_0| < T / (1-e^{-\lambda_0 T})^{{1}/{(2q)}}
    \atop 
    m\geq 0
  } \left| \nabla^m f(x) \right|
  <
 e^{
   x_0 + T(1-e^{-\lambda_0 T})^{-1/(2q)} }
 <
 e^{ x_0 + 1}. 
 $$
 Therefore, Conditions~\eqref{eqn-26} and \eqref{eqn-27} hold under
 \eqref{eqn-30-0-1} and %
$$
 T < 
 - \frac{1}{\lambda_0}
 \log \big( 1 - e^{-(x_0+1)(2q)}\big), 
$$
 i.e.
$$
 T < 
 \frac{2}{3} 
 \min\left(
 e^{-x_0-1},
 - 
 \log \big( 1 - e^{-(x_0+1)(2q)}\big)
 \right)
 < e^{-x_0}, 
$$
\end{example}
}
\appendix

\section{Branching trees {\em vs.} Butcher trees}
\label{btbt}
\noindent
 This section presents the full construction of the
 random tree $\mathcal{B}^{0,c}$, and describes the connection between 
 the functional  $\mathcal{H}_t \left(\mathcal{B}^{0,c}\right)$
 and the Butcher series \eqref{fjkld9} in Lemma~\ref{jklds1}.  
 Recall that a rooted tree $\tau = (V,E, \bigcdot)$ is a nonempty set $V$ of vertices and a set of edges $E$ between some of the pairs of vertices, with a specific vertex $\bigcdot$ called the root, such that the graph $(V,E)$ is connected with no loops.
We call the two special trees $\emptyset$ and $\bigcdot$ empty tree and dot tree respectively. We denote by $\mathbf{T}$ the set of all rooted
 trees,
 and by $\mathbf{T}_n$, $n\geq 0$, the subset of $\mathbf{T}$ consisting of all trees with order $n$.
 The following notion of grafting product is a generalization
 of the notion of beta product
 from unlabeled trees, see \cite[Section~2.1]{But21}, 
 to labeled trees. 

 \begin{definition}[Grafting product]
   \begin{enumerate}[i)]
     \item 
  Given $\tau_1$, $\tau_2$ two labeled trees and
  $l\in \{1,\ldots, |\tau_1|\}$,
  the grafting-product with label $l$ of $\tau_1$ and $\tau_2$,
  denoted by $\tau_1 *_l \tau_2$,
  is the tree of order
  $|\tau_1| + |\tau_2|$ formed by grafting (attaching) $\tau_2$
  from its root to the vertex $l$ of $\tau_1$,
  so that the whole $\tau_2$ are descendants of vertex $l$.
     \item 
  The new tree is labeled by keeping all labels of $\tau_1$,
  and by adding $|\tau_1|$ to all labels of $\tau_2$.
     \item 
  For any labeled tree $\tau$, we let
  $\emptyset *_0 \tau = \tau *_l \emptyset = \tau$ for all $0\le l\le |\tau|$, and keep the labels of $\tau$.
   \end{enumerate}
 \end{definition}
\noindent 
For $t\geq 0$ we consider the finite random tree
$\mathcal{B}^{0,\id}|_{[0,t]}$ obtained by killing all offsprings
of $\mathcal{B}^{0,\id}$ that are born beyond time $t$. 
The (random) leaves of $\mathcal{B}^{0,\id}|_{[0,t]}$ coincide with the set $\mathcal K^\pt_t$.
 Associated to every sample tree of $\mathcal{B}^{0,\id}|_{[0,t]}$, we
 draw a labeled and marked Butcher tree
 denoted by $\mathcal T^{\id}$ via the following 
 recursive algorithm based on the sequence of
 splitting times of $\mathcal{B}^{0,\id}|_{[0,t]}$ sorted
 in increasing order.\footnote{This is possible since
 the splitting times of $\mathcal{B}^{0,\id}|_{[0,t]}$ are a.s. distinct.}

 \newpage
  
\begin{algorithm}
 \caption{Recursive construction of the labeled tree $\mathcal T^{\id}$.} %
\label{fjkl432}
\begin{algorithmic}
\STATE \textbf{Initialization:}
Starting from an initial branch with
 label $0$ and mark $\id$,
 we initialize an empty tree $\emptyset$ with label $0$ and mark $\id$.

\vspace{-0.2cm}

\tikzstyle{level 1}=[level distance=4cm, sibling distance=4cm]
\tikzstyle{level 2}=[level distance=4cm, sibling distance=4cm]

\begin{table}[H]
\centering
\scalebox{0.8}{
  \begin{tabular}{||c|C{10cm}|C{3cm}|C{2.9cm}||}
    \hline Splitting 
    & Branching tree $\mathcal{B}^{0,\id}|_{[0,0]}$ %
    & Labeled tree $\mathcal T^{\id}$ %
    & $\prod_{\mathbf k \in \mathcal{K}^{\partial}_0} c_{\mathbf{k}}$ %
    \\
\hline \hline
0-th 
&
\scalebox{0.8}{
\begin{tikzpicture}[scale=0.7,grow=right, sloped]
\node[rectangle,draw,black,text=black,thick]{$0$}
    child {
        node[black,text=black,thick] {} %
            edge from parent
            node[above] {$0$}
            node[below]{${\rm Id}$}
    };
\end{tikzpicture}
}
&  
\begin{tikzpicture}[scale=0.7,grow=right, sloped]
\node[black,text=black,thick,scale=1.2]{$\prescript{}{0}{\varnothing}_{{\rm Id}}$};
\end{tikzpicture}
& $1$
\\
\hline\hline
\end{tabular}
}
\end{table}

\vspace{-0.3cm}

 If only one splitting time occurred before time $t$,
 \begin{itemize}
 \item the initial branch yields one offspring with
   label $1$ and mark $f$,
   and 
 \item
   we update the initial labeled tree $\emptyset$ to
 $\emptyset *_0 \bigcdot = \bigcdot$, and mark this root with $f$.
\end{itemize} 

\vspace{-0.3cm}

\begin{table}[H]
\centering
\scalebox{0.8}{
  \begin{tabular}{||c|C{10cm}|C{3cm}|C{2.9cm}||}
    \hline Splitting 
    & Branching tree $\mathcal{B}^{0,\id}|_{[0,t]}$ %
    & Labeled tree $\mathcal T^{\id}$ %
    & $\prod_{\mathbf k \in \mathcal{K}^{\partial}_t} c_{\mathbf{k}}$ %
    \\
\hline \hline
1st 
&
\scalebox{0.8}{
\begin{tikzpicture}[scale=0.7,grow=right, sloped]
\node[rectangle,draw,black,text=black,thick]{$0$}
    child {
        node[rectangle,draw,black,text=black,thick] {$T_1$}
            child {
                node[black,text=black, thick] {} %
                edge from parent
                node[above] {$1$}
                node[below]  {$f$}
            }
            edge from parent
            node[above] {$0$}
            node[below]{${\rm Id}$}
    };
\end{tikzpicture}
}
&  
 \scalebox{1}{
\begin{tikzpicture}
	\tikzmath{
		\LL = 1.25; %
		\RR = 0.07; %
	}
	
	\coordinate (4) at (0,0);
	
	\draw[fill=black] (4) circle(\RR) node[right]{f} node[left]{1};
\end{tikzpicture}
}
 & $f$
 \\
\hline\hline
\end{tabular}
}
\end{table}

\vspace{-0.3cm}

\STATE \textbf{Recursion:}
 Assume that a marked labeled tree $\tau$ with order $|\tau|=i \ge 1$  
 has been constructed, with $i$-$th$ splitting time not later than $t$. 
 \begin{itemize}
   \item 
 If the $(i+1)$-$th$ splitting time of a
 branch with label $l \geq 1$ and mark $\nabla^m f$, $m\geq 0$, 
 is not later than $t$ it then gives two offsprings, respectively
 with label $i+1$ and mark $f$,
 and with label $l$ and mark $\nabla^{m+1} f$. 
\item
  We update the labeled tree $\tau$ to $\tau *_l \bigcdot$. We also update the mark of the vertex $l$ to $\nabla^{m+1} f$ and mark the new vertex $i+1$ with $f$.
 \end{itemize}

 \vspace{-0.3cm}

\begin{table}[H]
\centering
\scalebox{0.8}{
  \begin{tabular}{||c|C{10cm}|C{3cm}|C{2.9cm}||}
    \hline Splitting 
    & Branching tree $\mathcal{B}^{0,\id}|_{[0,t]}$ %
    & Labeled tree $\mathcal T^{\id}$ %
    & $\prod_{\mathbf k \in \mathcal{K}^{\partial}_t} c_{\mathbf{k}}$ %
    \\
\hline \hline
2nd 
&
\scalebox{0.7}{
\begin{tikzpicture}[scale=0.7,grow=right, sloped]
\node[rectangle,draw,black,text=black,thick]{$0$}
    child {
        node[rectangle,draw,black,text=black,thick] {$T_1$}
            child {
                node[rectangle,draw,black,text=black,thick] {$T_2$}
                child{
                node[black,text=black,thick]{} %
                edge from parent
                node[above]{$1$}
                node[below]{$\nabla f$}
                }
                child[draw=blue]{
                node[thick]{} %
                edge from parent
                node[above]{$2$}
                node[below]{$\textcolor{blue}{f}$}
                }
                edge from parent
                node[above] {$1$}
                node[below]  {$f$}
            }
            edge from parent
            node[above] {$0$}
            node[below]{${\rm Id}$}
    };
\end{tikzpicture}
}
&
 \scalebox{1}{
\begin{tikzpicture}
	\tikzmath{
		\LL = 1.25; %
		\RR = 0.07; %
	}
	
	\coordinate (4) at (0,0);
	\coordinate (5) at ($(4) + (0, -\LL)$);
	
	\draw[thick] (4) -- (5);
		
	\draw[fill=black] (4) circle(\RR) node[right]{\textcolor{blue}{$f$}} node[left]{2};
	\draw[fill=black] (5) circle(\RR) node[right]{$\nabla f$} node[left]{1};
\end{tikzpicture}
}
 & $(\nabla f ) f$
 \\
\hline\hline
\end{tabular}
}
\end{table}

\vspace{-0.2cm}

\end{algorithmic}
\end{algorithm}

 \noindent 
 Since $\mathcal{B}^{0,\id}|_{[0,t]}$ has finite splitting, the above induction will end in finite steps. 
 The following graphs further illustrate this recursive
 construction.

\tikzstyle{level 1}=[level distance=4cm, sibling distance=4cm]
\tikzstyle{level 2}=[level distance=4cm, sibling distance=4cm]

\begin{table}[H]
\centering
\scalebox{0.8}{
  \begin{tabular}{||c|C{10cm}|C{3cm}|C{2.9cm}||}
\hline
Splitting 
& Branching tree $\mathcal{B}^{0,\id}|_{[0,t]}$ %
& Labeled tree $\mathcal T^{\id}$ %
& $\prod_{\mathbf k \in \mathcal{K}^{\partial}_t} c_{\mathbf{k}}$ %
\\
\hline \hline
3rd 
&
\scalebox{0.7}{
\begin{tikzpicture}[scale=0.7,grow=right, sloped]
\node[rectangle,draw,black,text=black,thick]{$0$}
    child {
        node[rectangle,draw,black,text=black,thick] {$T_0$}
            child {
                node[rectangle,draw,black,text=black,thick] {$T_{1}$}
                child{
                node[rectangle,draw,black,text=black,thick]{$T_2$}
                child{
                    node[black,text=black,thick]{} %
                    edge from parent
                    node[above]{$1$}
                    node[below]{$\nabla^2 f$}
                }
                child[draw=blue]{
                    node[black,text=black,thick]{} %
                    edge from parent
                    node[above]{$3$}
                    node[below]{$\textcolor{blue}{f}$}
                }
                edge from parent
                node[above]{$1$}
                node[below]{$\nabla f$}
                }
                child[draw=purple]{
                node[thick]{} %
                edge from parent
                node[above]{$2$}
                node[below]{$\textcolor{purple}{f}$}
                }
                edge from parent
                node[above] {$1$}
                node[below]  {$f$}
            }
            edge from parent
            node[above] {$0$}
            node[below]{${\rm Id}$}
    };
\end{tikzpicture}
}
&
 \scalebox{1}{
\begin{tikzpicture}
	\tikzmath{
		\LL = 1.25; %
		\RR = 0.07; %
	}
	
	\coordinate (4) at (0,0);
	\coordinate (5) at ($(4) + (-\LL / 2 , \LL)$);
	\coordinate (6) at ($(4) + (\LL / 2, \LL)$);
	
	\draw[thick] (4) -- (5);
	\draw[thick] (4) -- (6);
		
	\draw[fill=black] (4) circle(\RR) node[right]{$\nabla^2 f$} node[left]{1};
	\draw[fill=black] (5) circle(\RR) node[right]{\textcolor{purple}{$f$}} node[left]{2};
	\draw[fill=black] (6) circle(\RR) node[right]{\textcolor{blue}{$f$}} node[left]{3};
\end{tikzpicture}
}
 & $\nabla^2 f (f,f)$
 \\
\hline\hline
\end{tabular}
}
\end{table}

\vspace{-0.6cm}

\begin{table}[H]
\centering
\scalebox{0.8}{
\begin{tabular}{||C{1.4cm}|C{10cm}|C{3cm}|C{2.9cm}||}
\hline
\hline
3rd 
&
\scalebox{0.65}{
\begin{tikzpicture}[scale=0.8,grow=right, sloped]
\node[rectangle,draw,black,text=black,thick]{$0$}
    child {
        node[rectangle,draw,black,text=black,thick] {$T_0$}
            child {
                node[rectangle,draw,black,text=black,thick] {$T_{1}$}
                child{
                node[black,text=black,thick]{} %
                edge from parent
                node[above]{$1$}
                node[below]{$\nabla f$}
                }
                child{
                node[rectangle,draw,thick]{$T_2$}
                    child[draw=blue]{
                    node[black,text=black,thick]{} %
                    edge from parent
                    node[above]{$2$}
                    node[below]{$\textcolor{blue}{\nabla f}$}
                    }
                    child[draw=purple]{
                    node[black,text=black,thick]{} %
                    edge from parent
                    node[above]{$3$}
                    node[below]{$\textcolor{purple}{f}$}
                    }
                edge from parent
                node[above]{$2$}
                node[below]{$f$}
                }
                edge from parent
                node[above] {$1$}
                node[below]  {$f$}
            }
            edge from parent
            node[above]{$0$}
            node[below]{${\rm Id}$}
    };
\end{tikzpicture}
}
&
 \scalebox{1}{
\begin{tikzpicture}
	\tikzmath{
		\LL = 1.25; %
		\RR = 0.07; %
	}
	
	\coordinate (4) at (0,0);
	\coordinate (5) at ($(4) + (0, \LL)$);
	\coordinate (6) at ($(4) + (0, 2 * \LL)$);
	
	\draw[thick] (4) -- (5);
	\draw[thick] (4) -- (6);
		
	\draw[fill=black] (4) circle(\RR) node[right]{$\nabla f$} node[left]{1};
	\draw[fill=black] (5) circle(\RR) node[right]{\textcolor{blue}{$\nabla f$}} node[left]{2};
	\draw[fill=black] (6) circle(\RR) node[right]{\textcolor{purple}{$f$}} node[left]{3};
\end{tikzpicture}
}
 & $(\nabla f)^2f$
 \\
\hline\hline
\end{tabular}
}
\end{table}

\vspace{-0.6cm}

\begin{table}[H]
\centering
\scalebox{0.8}{
          {
            \begin{tabular}{||c|C{10cm}|C{3cm}|C{2.9cm}||}
\hline
Splitting 
& Branching tree $\mathcal{B}^{0,\id}|_{[0,t]}$ %
& Labeled tree $\mathcal T^{\id}$ %
& $\prod_{\mathbf k \in \mathcal{K}^{\partial}_t} c_{\mathbf{k}}$ %
\\ 
\hline \hline
4-th 
          &
\scalebox{0.6}{
\begin{tikzpicture}[scale=0.7,grow=right, sloped]
\node[rectangle,draw,black,text=black,thick]{$0$}
    child {
        node[rectangle,draw,black,text=black,thick] {$T_0$}
            child {
                node[rectangle,draw,black,text=black,thick] {$T_{1}$}
                child{
                node[rectangle,draw,black,text=black,thick]{$T_2$}
                    child{
                        node[rectangle,draw,black,text=black,thick]{$T_3$}
                        child{
                            node[black,text=black,thick]{} %
                        edge from parent
                        node[above]{$1$}
                        node[below]{$\nabla^3 f$}
                        }
                        child[draw=blue]{
                            node[black,text=black,thick]{} %
                        edge from parent
                        node[above]{$4$}
                        node[below]{$\textcolor{blue}{f}$}
                        }
                    edge from parent
                    node[above]{$1$}
                    node[below]{$\nabla^2 f$}
                    }
                    child[draw=purple]{
                        node[black,text=black,thick]{} %
                        edge from parent
                        node[above]{$3$}
                        node[below]{$\textcolor{purple}{f}$}
                    }
                edge from parent
                node[above]{$1$}
                node[below]{$\nabla f$}
                }
                child[draw=cyan]{
                node[thick]{} %
                edge from parent
                node[above]{$2$}
                node[below]{$\textcolor{cyan}{f}$}
                }
                edge from parent
                node[above] {$1$}
                node[below]  {$f$}
            }
            edge from parent
            node[above]{$0$}
            node[below]{${\rm Id}$}
    };
\end{tikzpicture}
       }
&
\scalebox{1}{
\begin{tikzpicture}
	\tikzmath{
		\LL = 1.25; %
		\RR = 0.07; %
	}
	
	\coordinate (4) at (0,0);
	\coordinate (5) at ($(4) + (- 4 * \LL /5 , \LL)$);
	\coordinate (6) at ($(4) + (0, \LL)$);
	\coordinate (7) at ($(4) + ( 4 * \LL /5 , \LL)$);
	
	\draw[thick] (4) -- (5);
	\draw[thick] (4) -- (6);
	\draw[thick] (4) -- (7);
		
	\draw[fill=black] (4) circle(\RR) node[right]{$\nabla^3 f$} node[left]{1};
	\draw[fill=black] (5) circle(\RR) node[right]{\textcolor{blue}{$f$}} node[left]{4};
	\draw[fill=black] (6) circle(\RR) node[right]{\textcolor{purple}{$f$}} node[left]{3};
	\draw[fill=black] (7) circle(\RR) node[right]{\textcolor{cyan}{$f$}} node[left]{2};
\end{tikzpicture}
}
& $(\nabla^3 f) (f,f,f)$
\\
\hline\hline
\end{tabular}
          }
          }
\end{table}

\vspace{-0.6cm}

\begin{table}[H]
\centering
\scalebox{0.8}{
\begin{tabular}{||C{1.4cm}|C{10cm}|C{3cm}|C{2.9cm}||}
\hline \hline
4-th 
&
\scalebox{0.6}{
\begin{tikzpicture}[scale=0.8,grow=right, sloped]
\node[rectangle,draw,black,text=black,thick]{$0$}
    child {
        node[rectangle,draw,black,text=black,thick] {$T_0$}
            child {
                node[rectangle,draw,black,text=black,thick] {$T_{1}$}
                child{
                node[rectangle,draw,black,text=black,thick,yshift=-0.4cm]{$T_3$}
                    child{
                        node[black,text=black,thick]{} %
                    edge from parent
                    node[above]{$1$}
                    node[below]{$\nabla^2 f$}
                    }
                    child[draw=blue]{
                        node[black,text=black,thick]{} %
                        edge from parent
                        node[above]{$4$}
                        node[below]{$\textcolor{blue}{f}$}
                    }
                edge from parent
                node[above]{$1$}
                node[below]{$\nabla f$}
                }
                child{
                node[rectangle,draw,thick,yshift=0.4cm]{$T_2$}
                    child[draw=purple]{
                    node[black,text=black,thick]{} %
                    edge from parent
                    node[above]{$2$}
                    node[below]{$\textcolor{purple}{\nabla f}$}
                    }
                    child[draw=cyan]{
                    node[black,text=black,thick]{} %
                    edge from parent
                    node[above]{$3$}
                    node[below]{$\textcolor{cyan}{f}$}
                    }
                edge from parent
                node[above]{$2$}
                node[below]{$f$}
                }
                edge from parent
                node[above] {$1$}
                node[below]  {$f$}
            }
            edge from parent
            node[above]{$0$}
            node[below]{${\rm Id}$}
    };
\end{tikzpicture}
       }
&
\scalebox{1}{
\begin{tikzpicture}
	\tikzmath{
		\LL = 1.25; %
		\RR = 0.07; %
	}
	
	\coordinate (4) at (0,0);
	\coordinate (5) at ($(4) + (- \LL /2 , \LL)$);
	\coordinate (6) at ($(4) + ( \LL /2 , \LL)$);
	\coordinate (7) at ($(4) + ( \LL /2, 2 * \LL)$);
	
	\draw[thick] (4) -- (5);
	\draw[thick] (4) -- (6);
	\draw[thick] (6) -- (7);
		
	\draw[fill=black] (4) circle(\RR) node[right]{$\nabla^2 f$} node[left]{1};
	\draw[fill=black] (5) circle(\RR) node[right]{\textcolor{blue}{$f$}} node[left]{4};
	\draw[fill=black] (6) circle(\RR) node[right]{\textcolor{purple}{$\nabla f$}} node[left]{2};
	\draw[fill=black] (7) circle(\RR) node[right]{\textcolor{cyan}{$f$}} node[left]{3};
\end{tikzpicture}
}
& $(\nabla^2 f) (f,(\nabla f ) f)$
\\
\hline\hline
\end{tabular}
}
\end{table}

\vspace{-0.6cm}

\begin{table}[H]
\centering
\scalebox{0.8}{
\begin{tabular}{||C{1.4cm}|C{10cm}|C{3cm}|C{2.9cm}||}
\hline \hline
4-th 
&
\scalebox{0.6}{
\begin{tikzpicture}[scale=0.8,grow=right, sloped]
\node[rectangle,draw,black,text=black,thick]{$0$}
    child {
        node[rectangle,draw,black,text=black,thick] {$T_0$}
            child {
                node[rectangle,draw,black,text=black,thick] {$T_{1}$}
                child{
                    node[black,text=black,thick]{} %
                edge from parent
                node[above]{$1$}
                node[below]{$\nabla f$}
                }
                child{
                node[rectangle,draw,thick]{$T_2$}
                    child{
                    node[rectangle,draw,black,text=black,thick]{$T_3$}
                    child[draw=blue]{
                        node[black,text=black,thick]{} %
                    edge from parent
                    node[above]{$2$}
                    node[below]{$\textcolor{blue}{\nabla^2 f}$}                    
                    }
                    child[draw=purple]{
                        node[black,text=black,thick]{} %
                    edge from parent
                    node[above]{$4$}
                    node[below]{$\textcolor{purple}{f}$}                    
                    }
                    edge from parent
                    node[above]{$2$}
                    node[below]{$\nabla f$}
                    }
                    child[draw=cyan]{
                    node[black,text=black,thick]{} %
                    edge from parent
                    node[above]{$3$}
                    node[below]{$\textcolor{cyan}{f}$}
                    }
                edge from parent
                node[above]{$2$}
                node[below]{$f$}
                }
                edge from parent
                node[above] {$1$}
                node[below]  {$f$}
            }
            edge from parent
            node[above]{$0$}
            node[below]{${\rm Id}$}
    };
\end{tikzpicture}
       }
&
\scalebox{1}{
  \begin{tikzpicture}
	\tikzmath{
		\LL = 1.25; %
		\RR = 0.07; %
	}
	
	\coordinate (4) at (0,0);
	\coordinate (5) at ($(4) + (0 , \LL)$);
	\coordinate (6) at ($(4) + ( - \LL /2 , 2* \LL)$);
	\coordinate (7) at ($(4) + ( \LL /2, 2 * \LL)$);
	
	\draw[thick] (4) -- (5);
	\draw[thick] (5) -- (6);
	\draw[thick] (5) -- (7);
		
	\draw[fill=black] (4) circle(\RR) node[right]{$\nabla f$} node[left]{1};
	\draw[fill=black] (5) circle(\RR) node[right]{\textcolor{blue}{$\nabla^2 f$}} node[left]{2};
	\draw[fill=black] (6) circle(\RR) node[right]{\textcolor{cyan}{$f$}} node[left]{3};
	\draw[fill=black] (7) circle(\RR) node[right]{\textcolor{purple}{$f$}} node[left]{4};
\end{tikzpicture}
}
& $(\nabla f) (\nabla^2 f) (f,f)$
\\
\hline\hline
\end{tabular}
}
\end{table}

\vspace{-0.6cm}

\begin{table}[H]
\centering
\scalebox{0.8}{
\begin{tabular}{||C{1.4cm}|C{10cm}|C{3cm}|C{2.9cm}||}
\hline \hline
4-th 
&
\scalebox{0.6}{
\begin{tikzpicture}[scale=0.8,grow=right, sloped]
\node[rectangle,draw,black,text=black,thick]{$0$}
    child {
        node[rectangle,draw,black,text=black,thick] {$T_0$}
            child {
                node[rectangle,draw,black,text=black,thick] {$T_{1}$}
                child{
                    node[black,text=black,thick]{} %
                edge from parent
                node[above]{$1$}
                node[below]{$\nabla f$}
                }
                child{
                node[rectangle,draw,thick]{$T_2$}
                    child[draw=blue]{
                    node[black,text=black,thick]{} %
                    edge from parent
                    node[above]{$2$}
                    node[below]{$\textcolor{blue}{\nabla f}$}
                    }
                    child{
                    node[rectangle,draw,black,text=black,thick]{$T_3$}
                        child[draw=purple]{
                            node[black,text=black,thick]{} %
                        edge from parent
                        node[above]{$3$}
                        node[below]{$\textcolor{purple}{\nabla f}$}
                        }
                        child[draw=cyan]{
                            node[black,text=black,thick]{} %
                        edge from parent
                        node[above]{$4$}
                        node[below]{$\textcolor{cyan}{f}$}
                        }
                    edge from parent
                    node[above]{$3$}
                    node[below]{$f$}
                    }
                edge from parent
                node[above]{$2$}
                node[below]{$f$}
                }
                edge from parent
                node[above] {$1$}
                node[below] {$f$}
            }
            edge from parent
            node[above]{$0$}
            node[below]{${\rm Id}$}
    };
\end{tikzpicture}
}
&
\scalebox{1}{
  \begin{tikzpicture}
	\tikzmath{
		\LL = 1.25; %
		\RR = 0.07; %
	}
	
	\coordinate (4) at (0,0);
	\coordinate (5) at ($(4) + (0 , \LL)$);
	\coordinate (6) at ($(4) + (0, 2* \LL)$);
	\coordinate (7) at ($(4) + (0, 3 * \LL)$);
	
	\draw[thick] (4) -- (5);
	\draw[thick] (5) -- (6);
	\draw[thick] (5) -- (7);
		
	\draw[fill=black] (4) circle(\RR) node[right]{$\nabla f$} node[left]{1};
	\draw[fill=black] (5) circle(\RR) node[right]{\textcolor{blue}{$\nabla f$}} node[left]{2};
	\draw[fill=black] (6) circle(\RR) node[right]{\textcolor{purple}{$\nabla f$}} node[left]{3};
	\draw[fill=black] (7) circle(\RR) node[right]{\textcolor{cyan}{$f$}} node[left]{4};
\end{tikzpicture}
}
& $(\nabla f )^3f$
\\
\hline\hline
\end{tabular}
}
\end{table}
        
\vspace{-0.4cm}

\noindent
 The above construction relies on two inductive assumptions,
 which are proved  in the next lemma
 by induction on splitting times. 
\begin{lemma}
  \label{l}
  \begin{enumerate}[i)]
  \item
    The labeled tree $\tau$ drawn at the $i$-$th$ splitting of $\mathcal{B}^{0,\id}|_{[0,t]}$ has order $|\tau| = i \geq 1$,
    its vertex labels form permutation of $\{0,\ldots,|\tau|\}$, 
    and each vertex with $m$ descendants is marked by $\nabla^m f$,
    with $m\le i-1$. 
\item Right after the $i$-$th$ splitting of $\mathcal{B}^{0,\id}|_{[0,t]}$, $i\ge1$, the labels of all living branches form a permutation of $\{0,\ldots, i\}$. %
  If the label of a living branch coincides with the label of a vertex of the tree $\tau$ that we have drawn, then their marks also coincide.
\end{enumerate}
\end{lemma}
\begin{proof}
  $i)$ 
  The statements clearly hold for the first splitting.
  Suppose that they hold for the tree $\tau$ drawn at the $i$-$th$
  splitting, $i\ge1$.
  In the $(i+1)$-$th$ splitting, the new tree
  $\tau *_l \bigcdot$ is a labeled tree with order
  $|\tau| +1 = i+1$.
  Compared to $\tau$, the only vertex whose descendant number changes in the new tree $\tau *_l \bigcdot$ is the vertex $l$. This vertex has mark $\nabla^m f$ in $\tau$, so by the inductive assumption, it has $m$ descendants with $m\le i$. In the new tree $\tau *_l \bigcdot$, the vertex $l$ has one more descendant than in $\tau$, and its mark is updated to $\nabla^{m+1} f$. %
  
  $ii)$ %
  The statements hold for the first splitting, since the only living branch right after this splitting must be the only child of the initial branch, which has label $1$ and mark $f$, which is the same as the mark of the root we have drawn. Suppose that the statements hold for the $i$-$th$ splitting, $i\ge1$. Then by construction, at the $(i+1)$-$th$ splitting, the branch with label $l$ dies and gives two offsprings are labeled by $i+1$ and $l$, while all other living branches have distinct labels forming a permutation of $\{0,\ldots, i\} \setminus \{l\}$. Thus, the labels of all these living branches have distinct labels forming a permutation of $\{0,\ldots, i+1\}$. Again by construction, the two offsprings with labels $i+1$ and $l$ have marks $f$ and $\nabla^{m+1} f$, which are the same as those of the vertex $i+1$ and $l$ in the new drawn tree, while all other living branches also share the same marks with the vertices whose labels equal to their labels, by the inductive assumption. This completes the proof.
\end{proof}
\noindent
 The next result is consequence of Lemma~\ref{l}. 
\begin{corollary}
  \label{c}
  \begin{enumerate}[i)]
  \item
     Each vertex of $\mathcal T^{\id}$
     having $m$ descendants has the mark $\nabla^m f$, 
     for $m = 1,\ldots , |\mathcal T^{\id}|-1$. 
\item 
  The labels of the leaves of $\mathcal{B}^{0,\id}|_{[0,t]}$ are distinct and
  they constitute a permutation of $\{0,\cdots, |\mathcal K^\pt_t|\}$. If the label of a leave of $\mathcal{B}^{0,\id}|_{[0,t]}$ coincides with the label of a vertex of $\mathcal T^{\id}$, then their marks also coincide.
  \end{enumerate} 
\end{corollary}
\noindent 
Next is the definition of elementary differentials,
cf. \cite[\S~3.3]{butcher2010}.
  \begin{definition}
 The elementary differential of $f$ is the mapping $F:\mathbf{T} \to \C^\infty(\R^d, \R^d)$ defined recursively by $F(\emptyset) = \id$, $F(\bigcdot) = f$, and $$
  F(\tau) = \nabla^m f(F(\tau_1), \ldots, F(\tau_m))$$
  for $\tau = [\tau_1, \ldots, \tau_m]$.
\end{definition}
  \noindent
   When  $|\tau|=n$, we have 
   \begin{equation}
     \nonumber %
  F(\tau ) =
  \nabla^{m_1} f(\nabla^{m_2} f(\cdots), \cdots, \cdots (\cdots, f) \cdots)
\end{equation}
for a sequence $(m_i)_{i=1,\ldots,n}$ of integers
satisfying $m_n=0$ and $\sum_{i=1}^n m_i = n-1$.
For a given tree $\tau\in\mathbf{T}$, the map $F$ also provides a way to mark each vertex of $\tau$ by $f$ or its derivatives: each vertex with no descendants is marked by $f$; the vertices with $m$ descendants
 are marked by $\nabla^m f$, for $m\geq 0$. 

\begin{lemma}
  \label{jklds1}
  Letting
  $\mathcal T^{\id}$ be the random labeled tree defined
  in Algorithm~\ref{fjkl432}, we have
  \begin{equation}
    \nonumber %
\prod_{\mathbf k \in \mathcal{K}^{\partial}_t} c_{\mathbf{k}} = F(\mathcal T^{\id}).
  \end{equation}
\end{lemma}
\begin{proof}
  By $ii)$ of Corollary~\ref{c}, the composition of marks over
  the leaves of $\mathcal{B}^{0,\id}|_{[0,t]}$ coincides with the
  composition of marks of all vertices of $\mathcal T^{\id}$.
  From the construction of the map $F$ 
  and $i)$ of Corollary~\ref{c},
  we conclude that the product
  $\prod_{\mathbf k \in \mathcal{K}^{\partial}_t} c_{\mathbf{k}}$   
  coincides with $F(\mathcal T^{\id})$.
\end{proof}

\paragraph{Acknowledgements.}
 This research is supported by the National Research Foundation, Singapore. %
 The work of Q. Huang is supported by the
Start-Up Research Fund of Southeast University
under Grant No. RF1028624194 and the
Jiangsu Provincial Scientific
Research Center of Applied Mathematics under Grant No. BK20233002.
 We thank an anonymous referee for useful comments. 

\footnotesize

\newcommand{\etalchar}[1]{$^{#1}$}
\def\cprime{$'$} \def\polhk#1{\setbox0=\hbox{#1}{\ooalign{\hidewidth
  \lower1.5ex\hbox{`}\hidewidth\crcr\unhbox0}}}
  \def\polhk#1{\setbox0=\hbox{#1}{\ooalign{\hidewidth
  \lower1.5ex\hbox{`}\hidewidth\crcr\unhbox0}}} \def\cprime{$'$}

\end{document}